\documentclass[a4paper]{amsart}
\usepackage[all]{xy}
\usepackage{amssymb}
\usepackage{amsmath}
\usepackage{bm}
\usepackage[colorlinks,linkcolor=blue]{hyperref}
\usepackage{color,xcolor,dsfont}
\usepackage{enumerate}
\usepackage{epsfig}
\usepackage{extarrows}
\usepackage{framed}
\usepackage{graphicx}
\usepackage{marvosym}
\usepackage{mathrsfs}
\usepackage{setspace} 
\usepackage{tikz}
\usepackage{tikz-cd}

\newcommand{\agit}{\,/\!\!/\,}
	
\newcommand{\git}{\mathbin{
		\mathchoice{\backslash\mkern-6mu\backslash}
		{\backslash\mkern-6mu\backslash}
		{\backslash\mkern-5mu\backslash}
		{\backslash\mkern-5mu\backslash}}}
	
\newcommand{\scha}{\mathcal{A}}
\newcommand{\schb}{\mathcal{B}}
\newcommand{\schc}{\mathcal{C}}
\newcommand{\schd}{\mathcal{D}}
\newcommand{\sche}{\mathcal{E}}

\newcommand{\schh}{\mathcal{H}}
\newcommand{\schk}{\mathcal{K}}
\newcommand{\schm}{\mathcal{M}}
\newcommand{\schp}{\mathcal{P}}
\newcommand{\schq}{\mathcal{Q}}
\newcommand{\schs}{\mathcal{S}}
\newcommand{\scht}{\mathcal{T}}
\newcommand{\schu}{\mathcal{U}}

\newcommand{\shf}{\mathscr{F}}
\newcommand{\shg}{\mathscr{G}}
\newcommand{\shh}{\mathscr{H}}
\newcommand{\shi}{\mathscr{I}}

\newcommand{\shl}{\mathscr{L}}

\newcommand{\sho}{\mathscr{O}}

\newcommand{\shu}{\mathscr{U}}

\DeclareMathOperator{\deriver}{R}

\DeclareMathOperator{\stable}{st}
\DeclareMathOperator{\modulis}{M}

\newcommand{\HH}{\textup{H}}
\newcommand{\Mukai}{\tilde{\HH}}
\newcommand{\MukaiZ}[1]{\Mukai({#1},\bz)}
\newcommand{\Mst}{\textup{\textsf{M}}}
\newcommand{\Ala}{\textup{\textsf{A}}}
\newcommand{\ExL}{\textup{\textsf{E}}}
\newcommand{\HyU}{\textup{\textsf{U}}}
\DeclareMathOperator{\topo}{top}

\DeclareMathOperator{\OO}{O}
\DeclareMathOperator{\Bl}{Bl}
\DeclareMathOperator{\id}{id}
\DeclareMathOperator{\NL}{NL}
\DeclareMathOperator{\NS}{NS}
\DeclareMathOperator{\Gr}{Gr}
\DeclareMathOperator{\CGr}{CGr}
\DeclareMathOperator{\Pic}{Pic}
\DeclareMathOperator{\Aut}{Aut}

\DeclareMathOperator{\Hom}{Hom}
\DeclareMathOperator{\Ext}{Ext}

\DeclareMathOperator{\PSL}{PSL}

\DeclareMathOperator{\Stab}{Stab}
\DeclareMathOperator{\rank}{rank}
\DeclareMathOperator{\coho}{coho}
\DeclareMathOperator{\disc}{disc}
\DeclareMathOperator{\Cone}{Cone}

\newcommand{\bz}{\mathbb{Z}}
\newcommand{\bq}{\mathbb{Q}}
\newcommand{\br}{\mathbb{R}}
\newcommand{\bc}{\mathbb{C}}

\newcommand{\bp}{\mathbb{P}}
\newcommand{\bh}{\mathbb{H}}

\newcommand{\Kum}{\textup{\textsf{K}}}
\newcommand{\Num}{\textup{\textsf{N}}}
\newcommand{\HR}{\textup{\textsf{HR}}}
\newcommand{\ch}{\textup{\textsf{ch}}}
\newcommand{\td}{\textup{\textsf{td}}}
\newcommand{\Stwo}{\mathfrak{S}_2}
\newcommand{\DbS}{\cate{D}^b(S)}
\newcommand{\DbX}{\cate{D}^b(X)}

\DeclareMathOperator{\KK}{K3}
\DeclareMathOperator{\BE}{BE}
\DeclareMathOperator{\Hilb}{Hilb}
\newcommand{\MBE}{\mathfrak{M}_{\BE}}
\newcommand{\MKK}{\mathfrak{M}_{\KK}}

\newcommand{\GLp}{\operatorname{GL^+}}
\newcommand{\grp}{{\tilde{\GLp}}(2,\br)}

\newcommand{\defi}[1]{{\textbf{\emph{#1}}}}
\newcommand{\cate}[1]{{\textbf{\textup{#1}}}}
\newcommand{\funct}[1]{{\textup{\textbf{\textsf{#1}}}}}

\theoremstyle{plain}
\newtheorem{theorem}{Theorem}[section]
\newtheorem{corollary}[theorem]{Corollary}
\newtheorem{lemma}[theorem]{Lemma}
\newtheorem{proposition}[theorem]{Proposition}

\theoremstyle{definition}
\newtheorem{definition}[theorem]{Definition}
\newtheorem{remark}[theorem]{Remark}
\newtheorem{example}[theorem]{Example}
\newtheorem{conjecture}[theorem]{Conjecture}

\author{Ziqi Liu}
\title{Bridgeland--Enriques general K3 surfaces}
\address{Dipartimento di Matematica “F. Enriques”, Università degli Studi di Milano, Via Cesare Saldini 50, 20133 Milano, Italy.}
\email{ziqi.liu@unimi.it}



\thanks{The author is a member of the INdAM group GNSAGA (2025-2026) and is partly supported by the ERC Grant 2017-CoG-771507 and FIS 3 Starting Grant No.~FIS-2024-04715.}

\begin{document}

\begin{abstract}
This article introduces a notion of Bridgeland--Enriques general K3 surfaces motivated by the study of Enriques categories over K3 surfaces and the invariant Bridgeland stability conditions. The family of Bridgeland--Enriques general K3 surfaces of degree $10$ detects a categorical degeneration of special Gushel--Mukai threefolds. Also, the families of Bridgeland--Enriques general K3 surfaces with higher degrees are closely related to Hodge-special Gushel--Mukai fourfolds and double EPW sextics.
\end{abstract}

\maketitle

\section{Introduction}
This article introduces the notion of a Bridgeland--Enriques general K3 surface motivated by the study on non-trivial derived autoequivalences of Picard rank one K3 surfaces \cite{FL23,Liu26b} and on two families of very general K3 surfaces \cite{Liu26a}. 

The term \emph{Bridgeland} comes from the notion of a (Bridgeland) \emph{stability condition} introduced by Bridgeland \cite{Bri07,Bri08}. The set of certain stability conditions on $\cate{D}^b(S)$ carries a complex manifold structure and is denoted by $\Stab^{\dag}(S)$.

The term \emph{Enriques} comes from the notion of an \emph{Enriques category} introduced by Bayer and Perry \cite{BP23}. An Enriques category $\sche$ over a K3 surface $S$ is the equivariant category $\cate{D}^b(S)_{\Stwo}$ with respect to an $\Stwo$-action on $\cate{D}^b(S)$ induced by a non-symplectic involution on $\cate{D}^b(S)$, see \cite[Appendix A]{Liu26a}.

The connection between stability conditions and Enriques categories starts from the Picard rank one K3 surfaces. It is shown in \cite{FL23} that any non-trivial derived involution $\Pi$ on a Picard rank one K3 surface $S$ is non-symplectic and fixes a stability condition in $\Stab^{\dag}(S)$. In this article, we show that $\Pi$ induces an $\Stwo$-action on $\cate{D}^b(S)$ and therefore gives an Enriques category using \cite{BO23,Liu26a}. 

\subsection{Bridgeland--Enriques general K3 surfaces}
The connection can in fact be extended to very general polarized K3 surfaces. 

Let $\MKK(2t)$ be the moduli space of polarized K3 surfaces of degree $2t$, then for any element $(S,H)$ in $\MKK(2t)$ one has a marking
$$\iota_{(S,H)}\colon \HyU\oplus\bz(2t)=:\Mst_{2t}\stackrel{\sim}{\rightarrow}\HH^0(S,\bz)\oplus\bz[H]\oplus\HH^4(S,\bz)$$
which sends the natural basis of the hyperbolic lattice $\HyU$ to $(1,0,0)$ and $(0,0,-1)$, and the positive generator of $\bz(2t)$ to $(0,H,0)$.

\begin{definition}\label{00_main-definition}
A nonempty open subset $B\subset\MKK(2t)$ is said to be \emph{Bridgeland--Enriques general} if there exists an involutive isometry $\tau$ on $\Mst_{2t}$ which extends to a non-trivial involutive Hodge isometry $\tilde{\tau}$ on $\MukaiZ{S}$ through the marking $\iota_{(S,H)}$ for any polarized K3 surface $(S,H)$ in $B$ such that
\begin{enumerate}
	\item there exists a non-trivial involution $\Pi$ on $\cate{D}^b(S)$ which has the Hodge isometry $\tilde{\tau}$ and fixes a stability condition in $\Stab^{\dag}(S)$;
	\item the involution $\Pi$ induces an $\Stwo$-action on $\cate{D}^b(S)$ such that the equivariant category $\cate{D}^b(S)_{\Stwo}$ is an Enriques category.
\end{enumerate}
A polarized K3 surface of degree $2t$ is said to be \defi{Bridgeland--Enriques general} if it belongs to such an open subset. The open locus $\MBE(2t)$ of all such polarized K3 surfaces is called the \defi{Bridgeland--Enriques general locus}.
\end{definition}

A very general member $(S,H)$ in $\MBE(2t)$ has Picard rank one. So one can see from Theorem \ref{02_FL-result} that $\MBE(2t)\neq\varnothing$ if and only if $4$ does not divide $t$ and any odd prime factor of $t$ has the form $4k+1$. 

\begin{example}[Proposition \ref{11_first-example}]\label{00_example-t=1}
One has $\MBE(2)=\MKK(2)$.
\end{example}

In general, the non-empty Bridgeland--Enriques general loci can be described in terms of the \emph{Noether--Lefschetz divisors} in $\MKK(2t)$. One recalls that a polarized K3 surface $(S,H)$ in $\MKK(2t)$ is contained in the Noether--Lefschetz divisor $\NL^{2t}_{n,2d}$ if and only if it has a divisor class $D$ with $D\cdot H=n$ and $D^2=2d$.

\begin{theorem}[Theorem \ref{11_main-result-1}]\label{00_main-result-1}
Consider an integer $t\geq 2$ such that $\MBE(2t)$ is non-empty, then $\MBE(2t)$ is the union of open subsets $\MBE(2t)_a$ indexed by integers $1\leq a\leq t$ with $a^2+1=mt$ for some $m\in\bz_+$ and defined by
$$\MBE(2t)_a:=\MKK(2t)-\bigcup_{(n,d)\in I_{t,a}}\NL_{n,2d}^{2t}$$
where $I_{t,a}=\{(2ar,mr^2-1)\mid r\in\bz_+\textup{ and }r^2<t\}$ is a finite set of integer pairs.
\end{theorem}
 
\begin{remark}
A polarized K3 surface in $\MKK(2t)$ with Picard rank one is contained in any non-empty open locus $\MBE(2t)_a$. Then one has finitely many distinct involutions on $\DbS$. One can check that these involutions cover all the conjugacy classes of involutions on $\DbS$ using \cite{FL23,Liu26b}. In particular, we classify all the Enriques categories over general K3 surfaces.
\end{remark}

In fact, the open locus $\MBE(2t)_a$ is exactly the subset of $\MKK(2t)$ where one has a uniformly defined stability condition $\sigma_{t,a}$ on $\DbS$ with central charge
$$Z_{t,a}(r,\Delta,s)=-\left(\frac{a^2-1}{t}r+\frac{a}{t}\Delta\cdot H+s\right)+\left(\frac{2ar+\Delta\cdot H}{t}\right)\bm{i}$$
for the polarized K3 surface $(S,H)$ of degree $2t$. 

Moreover, the open locus $\MBE(2t)_a$ is canonical also in the following sense.

\begin{theorem}[Theorem \ref{11_main-result-1.2}]\label{00_main-result-1.2}
Consider an integer $t\geq 2$ and $1\leq a\leq t$ such that the open locus $\MBE(2t)_a$ is non-empty, then for any family $p\colon\schs\rightarrow B$ of polarized K3 surfaces over a smooth connected variety with fibers in $\MBE(2t)_a$ one can define an autoequivalence $\Pi$ on the category of perfect complexes $\cate{D}_p(\schs)$, such that $\Pi^2\cong p^*\shl$ for some line bundle $\shl$ on $B$ and the base change of $\Pi$ for each point $b\rightarrow B$ induces the $\Stwo$-action on $\cate{D}^b(\schs_b)$ in Definition \ref{00_main-definition} and fixes $\sigma_{t,a}$.
\end{theorem}

\begin{remark}
In particular, $\Pi$ locally induces an $\Stwo$-action on $\cate{D}_p(\schs)$ by \cite{BP23} and locally fixes a stability condition on $\cate{D}_p(\schs)$ by \cite{BLMNPS21}. The local action can be made global for families of K3 surfaces in $\MBE(2t)_a$ by the upcoming work \cite{BPPZ26}.
\end{remark}

The first few $t$'s such that $\MBE(2t)\neq\varnothing$ are $1,2,5,10,13,17,25$. One will see that the open locus $\MBE(4)$ is the locus of quartic K3 surfaces $(S,H)$ in $\MKK(4)$ such that $\sho_S(H)\cong\sho_{\bp^3}(1)|_S$, and there exists a universal family $\schs$ of quartic K3 surfaces together with a uniform $\Stwo$-action on each fiber $\cate{D}^b(\schs_b)$. 

According to \cite{KP17}, one has the identification
$$\DbS_{\Stwo}\cong\schb_Y$$
where $\schb_Y:=\langle\sho_Y,\sho_Y(1)\rangle^{\perp}\subset\cate{D}^b(Y)$ is the Kuznetsov component of the double cover $f\colon Y\rightarrow\bp^3$ branched over the quartic surface $S$ with $\sho_Y(1)=f^*\sho_{\bp^3}(1)$. The threefold $Y$ is usually called the \emph{quartic double solid} over $S$.

\subsection{The categorical degeneration}
The open locus $\MBE(10)$ already contains some interesting information. One can see
$$\MBE(10)=\MKK(10)-\NL_{4,0}^{10}$$
using Theorem \ref{00_main-result-1}. In particular, one has

\begin{proposition}[Proposition \ref{12_main-result-2.1}]\label{00_main-result-2.1}
A polarized K3 surface $(S,H)$ in $\MBE(10)$ belongs to one of the following two cases:
\begin{enumerate}
	\item\label{00_main-result-2.1(1)} a Brill--Noether general surface of degree $10$ which is not in $\NL_{4,0}^{10}$;
	\item\label{00_main-result-2.1(2)} a quartic K3 surface containing an elliptic curve $E$ such that $H\cdot E=3$.
\end{enumerate}
Conversely, any polarized K3 surface in (1) is contained in $\MBE(10)$, and a polarized K3 surface in (2) is in $\MBE(10)$ if one has $\sho_S(H-E)\cong\sho_{\bp^3}(1)|_{S}$.
\end{proposition}

Moreover, there exists a unique Bridgeland--Enriques general branch in $\MBE(10)$ say $\MBE(10)_2$. So any polarized K3 surface $(S,H)$ in $\MBE(10)$ is associated with a unique Enriques category $\DbS_{\Stwo}$.

A surface $S$ in Proposition \ref{00_main-result-2.1}~\eqref{00_main-result-2.1(1)} is called a \emph{strongly smooth Gushel--Mukai surface} as in \cite{Beri25,DK18}, and embeds into a smooth threefold $M=\Gr(2,5)\cap\bp^6$. The \emph{special Gushel--Mukai threefold} $X$ over $S$ is the double cover $X\rightarrow M$ branched over $S$, and one can define a subcategory $\scha_X\subset\DbX$ called the \emph{Kuznetsov component} of $X$ (see Definition \ref{12_Kuz-special-GM-3-defi}). Again due to \cite{KP17}, one has $\DbS_{\Stwo}\cong\scha_X$.

On the other hand, one has

\begin{theorem}[Theorem \ref{12_main-result-2.2}]\label{00_main-result-2.2}
Consider a K3 surface $S$ in Proposition \ref{00_main-result-2.1}~\eqref{00_main-result-2.1(2)}, then the $\Stwo$-action on $\DbS$ coming from $(S,H)\in\MBE(10)$ is conjugate to the $\Stwo$-action on $\DbS$ coming from $(S,D)\in\MBE(4)$ for $D:=H-E$. 

In particular, the Enriques category $\DbS_{\Stwo}$ is equivalent to the Kuznetsov component $\schb_Y$ of the quartic double solid $Y$ over $S$.
\end{theorem}

\begin{remark}
A similar statement has been established in \cite[Proposition 8.9]{BP23} for K3 surfaces in Proposition \ref{00_main-result-2.1}~\eqref{00_main-result-2.1(2)} with Picard rank two.
\end{remark}

Also, it is shown in \cite{KS25} that one can canonically associate the quartic double solid $Y$ with a degenerate Gushel--Mukai threefold $X_{\star}$ such that $\schb_Y$ is equivalent to a subcategory $\bar{\scha}_{X_{\star}}\subset\cate{D}^b(X_{\star})$ called the \emph{Kuznetsov component} of $X_{\star}$. 

In conclusion, one can find families of Enriques categories
$$\cate{D}_p(\schs_B)_{\Stwo}\rightarrow B$$
whose fiber $\cate{D}^b(\schs_b)_{\Stwo}$ at $b\in B$ is either equivalent to the Kuznetsov component of a special Gushel--Mukai threefold or that of a degenerate one.

\begin{remark}
It would be interesting to prove the equivalence
$$\schb_Y\cong\cate{D}^b(S)_{\Stwo}\cong\bar{\scha}_{X_{\star}}$$ 
for the quartic double solid $Y$ over a K3 surface $S$ in Proposition \ref{00_main-result-2.1}~\eqref{00_main-result-2.1(2)} directly using the equivariant categories together with the geometry of $S$ and $X_{\star}$. 
\end{remark}

\subsection{Towards Hodge-special Gushel--Mukai fourfolds}
It also turns out that the open locus $\MBE(2t)$ for $t\geq 5$ is related to another important family of 2-Calabi--Yau categories: the Kuznetsov components of Gushel--Mukai fourfolds.

A \emph{Gushel--Mukai fourfold} is a smooth four-dimensional intersection
$$X=\CGr(2,V_5)\cap\bp^8\cap Q$$
where $Q\subset\bp^8\subset\bp(\bc\oplus\land^2V_5)$ is a quadric hypersurface. One can define for any $X$ a 2-Calabi--Yau category $\schc_X$ called the \emph{Kuznetsov component of $X$} together with an $\Stwo$-group action on $\schc_X$ such that $(\schc_X)_{\Stwo}$ is an Enriques category \cite{KP17,KP18}. 

A Gushel--Mukai fourfold $X$ is called \emph{Hodge-special} when $\HH^{2,2}(X,\bz)$ contains a rank-three primitive sublattice $\Lambda$ containing $\gamma_X^*\HH^4(\Gr(2,V_5),\bz)$ for the Gushel map $\gamma_X\colon X\rightarrow\Gr(2,V_5)$. The \emph{discriminant} of a Hodge-special $X$ is defined to be the discriminant $\disc(\Lambda)=2t$, where $t\geq 5$ and $t\equiv 0,1,2\pmod{4}$ due to \cite{DIM15}. 

In this article, we will call $X$ a \emph{first-type Hodge-special Gushel--Mukai fourfold} when its discriminant $2t$ satisfies $t\equiv1,2\pmod{4}$, and a second-type otherwise. 

\begin{proposition}[Proposition \ref{13_main-result-3.1}]\label{00_main-result-3.1}
Consider a Gushel--Mukai fourfold $X$ such that there exists an equivalence $\schc_X\cong\DbS$ for some Bridgeland--Enriques general K3 surface $S$ that is compatible with the $\Stwo$-actions, then $X$ is a first-type Hodge-special Gushel--Mukai fourfold.
\end{proposition}

Conversely, one has

\begin{theorem}[Theorem \ref{13_main-result-3.2}]\label{00_main-result-3.2}
Consider a first-type Hodge-special Gushel--Mukai fourfold $X$ with $\rank \HH^{2,2}(X,\bz)=3$ and discriminant $2t$, then for any non-empty open locus $\MBE(2t)_a$ one can find a K3 surface $S$ with Picard rank one in it such that there exists an equivalence $\schc_X\cong\DbS$ compatible with the $\Stwo$-actions.
\end{theorem}

\begin{remark}
Compared with \cite[Proposition 4.3]{BP21}, these K3 surfaces cover all the Fourier--Mukai partners of $X$.
\end{remark}

The second-type Hodge-special Gushel--Mukai fourfolds are related to twisted K3 surfaces via \cite{BP21,Per19}. One should be able to formulate a notion of Bridgeland--Enriques general twisted K3 surfaces for a twisted version of Proposition \ref{00_main-result-3.1} as well as Theorem \ref{00_main-result-3.2}. It is more difficult due to the phenomena discussed in \cite[Section 3.3]{Per19} and the lack of a clean twisted version of \cite{FL23,Hu16,Ka14}. 

On the other hand, Theorem \ref{00_main-result-3.2} already has some useful implications and brings up some difficult questions.

\subsection{Moduli spaces and double EPW sextics}
At first, an $\Stwo$-equivariant equivalence $\schc_X\cong\DbS$ for some $S$ in $\MBE(2t)_a$ will ensure a sublattice $\Ala_1^2$ in the algebraic part of the Mukai lattice $\MukaiZ{S}$ for any $S$ in $\MBE(2t)_a$. One can work out the generators say $\bar{v}_1$ and $\bar{v}_2$ explicitly in terms of $t$ and $a$.

In this case, the $\Stwo$-action on $\cate{D}^b(S)$ for any polarized K3 surface $(S,H)$ in $\MBE(2t)_a$ fixes the two moduli spaces $\modulis_{\sigma_{t,a}}(\bar{v}_1)$ and $\modulis_{\sigma_{t,a}}(\bar{v}_2)$ of semistable objects according to Theorem \ref{00_main-result-1.2}. The following statement is not hard.

\begin{proposition}[Proposition \ref{13_main-result-4.1}]\label{00_main-result-4.1}
Consider a very general K3 surface $S$ in some open locus $\MBE(2t)_a$ for $t\geq10$, then the moduli spaces $\modulis_{\sigma_{t,a}}(\bar{v}_1)$ and $\modulis_{\sigma_{t,a}}(\bar{v}_2)$ are $\KK^{[2]}$-type hyperkähler manifolds of degree $2$.
\end{proposition}

According to \cite{DK18}, a Gushel--Mukai fourfold $X$ is canonically associated with two irreducible symplectic varieties, the double EPW sextic $\tilde{Y}_X$ and double dual EPW sextic $\tilde{Y}_X^\vee$. These varieties are $\KK^{[2]}$-type hyperkähler manifolds with degree $2$ when they are smooth. Moreover, they admit natural non-symplectic involutions whose invariant lattices are exactly the rank one sublattice spanned by the ample class.

Using this information and our understanding of $\modulis_{\sigma_{t,a}}(\bar{v}_i)$, one can establish the following statement.

\begin{theorem}[Theorem \ref{13_main-result-4.2}]\label{00_main-result-4.2}
Consider an equivariant equivalence $\schc_X\cong\DbS$ for some Gushel--Mukai fourfold $X$ and some Picard rank one K3 surface $S$ in $\MBE(2t)_a$ with $t\geq 10$ such that the double EPW varieties $\tilde{Y}_X$ and $\tilde{Y}^\vee_X$ are smooth, then $\modulis_{\sigma_{t,a}}(\bar{v}_i)$ is isomorphic to $\tilde{Y}_X$ or $\tilde{Y}^\vee_X$ for $i=1,2$. 
\end{theorem}

\begin{remark}
It extends \cite[Proposition 5.17]{PPZ22}. Also, it is expected that the assertion can be generalized for any equivariant equivalence $\schc_X\cong\DbS$.
\end{remark}

The valuable point of our modular description for the double EPW sextics is the explicitness. Since the moduli spaces $\modulis_{\sigma_{t,a}}(\bar{v}_1)$ and $\modulis_{\sigma_{t,a}}(\bar{v}_2)$ are constructed uniformly in the locus $\MBE(2t)_a$, it is possible to spread the double EPW sextic locus in the moduli space $\mathfrak{M}_{\KK^{[2]}}(2)$ of $\KK^{[2]}$-type hyperkähler manifolds with a degree $2$ polarization using Theorem \ref{00_main-result-4.2} and \cite[Theorem 1.1~(2)]{O'Grady06}. 

\begin{conjecture}\label{00_conjecture-1}
Let $S$ be a K3 surface in $\MBE(2t)$ with $t\geq 5$ such that there is no K3 surface $S'$ in $\MBE(4)$ satisfying $\cate{D}^b(S)\cong\cate{D}^b(S')$, then $\modulis_{\sigma_{t,a}}(\bar{v}_i)$ is a double EPW sextic for $i=1,2$ and some $a$ satisfying $a^2\equiv-1\pmod{t}$.
\end{conjecture}

\begin{remark}
This conjecture is partially verified in \cite{LZ26} for K3 surfaces in Proposition \ref{00_main-result-2.1}~\eqref{00_main-result-2.1(1)} using \cite{DK24}. Also, the conjecture is partly checked for a general K3 surface in $\MBE(2t)_a$ with $a^2-t=-1$ and $t\geq 10$ by \cite[Corollary 7.6]{DM19}. 
\end{remark}

It is related to the conjecture on the image of the period map for Gushel--Mukai fourfolds, see \cite[Question 9.1]{DIM15}. Moreover, one has a stronger conjecture.

\begin{conjecture}
Let $S$ be a K3 surface in $\MBE(2t)$ with $t\geq 10$ such that there is no K3 surface $S'$ in $\MBE(4)\cup\MBE(10)$ satisfying $\cate{D}^b(S)\cong\cate{D}^b(S')$, then there exists a first-type Hodge-special Gushel--Mukai fourfold $X$ such that one has an $\Stwo$-equivariant equivalence $\cate{D}^b(S)\cong\schc_X$.
\end{conjecture}

In general, it would be interesting to decide how many Gushel--Mukai fourfolds one can find for a given Bridgeland--Enriques general K3 surface $S$.

\subsection{Organization}
In Section \ref{000_Section-02}, we recall necessary facts about the bounded derived categories of K3 surfaces. In Section \ref{000_Section-11}, we formulate Bridgeland--Enriques general K3 surfaces and make some observations on their loci. In Section \ref{000_Section-12}, we discuss the relation between Bridgeland--Enriques general K3 surfaces of degree $10$ and a categorical degeneration of special Gushel--Mukai threefolds. In Section \ref{000_Section-13}, we explain the connection with Hodge-special Gushel--Mukai fourfolds and associated double EPW sextics. In Appendix \ref{000_Section-A}, we develop some general theory about Enriques categories over K3 surfaces and in particular prove Proposition \ref{A_ppst}.

\subsection{Conventions}
This article works over the field $\bc$ of complex numbers. The imaginary unit is denoted by $\bm{i}$ or $\sqrt{-1}$. A variety is an irreducible and reduced quasi-projective scheme. A variety $X$ is said to be $1$-nodal if its singular locus consists of a single ordinary double point. A K3 surface is smooth and projective. A polarization on a K3 surface is assumed to be primitive.

\section{Derived Categories of K3 surfaces}\label{000_Section-02}
In this section, we will discuss the bounded derived category $\DbS$ of a K3 surface $S$ with a focus on a distinguished connected component $\Stab^{\dag}(S)$ of stability conditions on $\DbS$ and the group $\Aut(\DbS)$ of autoequivalences.

\subsection{The Mukai lattice and derived categories}
Let $S$ be a K3 surface, then there exists a canonical weight-two Hodge structure on 
$$\MukaiZ{S}:=\HH^0(S,\bz)\oplus \HH^2(S,\bz)\oplus \HH^4(S,\bz)$$
such that $\Mukai\,\!^{1,1}(S)=\HH^{0,0}(S)\oplus \HH^{1,1}(S)\oplus \HH^{2,2}(S)$ and $\Mukai\,\!^{2,0}(S)=\HH^{2,0}(S)$. This Hodge structure carries a polarization defined by the Mukai pairing
$$\langle(r_1,\alpha_1,s_1),(r_2,\alpha_2,s_2)\rangle=\alpha_1.\alpha_2-r_1.s_2-r_2.s_1$$
where the products are usual multiplications of cohomological classes. 

Usually, we will simply denote $v^2:=\langle v,v\rangle$ for a Mukai vector $v\in\MukaiZ{S}$.

\begin{theorem}[Orlov \cite{Or97}]\label{02_K3-Orlov}
An (exact) autoequivalence $\Phi$ on $\DbS$ is isomorphic to a Fourier--Mukai transform $\Phi_K$ for a uniquely determined $K\in\cate{D}^b(S\times S)$, and one has a group homomorphism
$$\rho\colon\Aut(\DbS)\rightarrow\OO(\MukaiZ{S}),\quad \Phi_K\mapsto\Phi^{\coho}_K$$
defined by sending $\Phi_K$ to the associated cohomological Fourier--Mukai map.
\end{theorem}

In analogy with the notions of symplectic and non-symplectic automorphisms on K3 surfaces, one has the following definitions.

\begin{definition}
An autoequivalence $\Phi$ on $\DbS$ is called \defi{symplectic} if the Hodge isometry $\phi=\rho(\Phi)$ preserves the symplectic class in $\Mukai\,\!^{2,0}(S)$. Otherwise, the auto-equivalence is said to be \defi{non-symplectic}.
\end{definition}

The Mukai vector of an object $E$ in $\DbS$ is defined as
$$v(E)=\ch(E).\sqrt{\td(S)}=(\rank(E),c_1(E),\rank(E)+\frac{c_1(E)^2}{2}-c_2(E))$$
so that $\chi(E_1,E_2)=-\langle v(E_1),v(E_2)\rangle$ for any objects $E_1,E_2$ in $\DbS$. It realizes
$$\Num(S)=\HH^0(S,\bz)\oplus \NS(S)\oplus \HH^4(S,\bz)$$
as a sub-Hodge-space of the polarized weight-two Hodge space $\MukaiZ{S}$. The group homomorphism $\rho$ in Theorem \ref{02_K3-Orlov} then restricts to $\Aut(\DbS)\rightarrow\OO(\Num(S))$.

The lattice $\MukaiZ{S}$ has signature $(4,20)$ under the Mukai pairing, and one has a group $\OO^+(\MukaiZ{S})$ containing the Hodge isometries of $\MukaiZ{S}$ preserving an orientation of a positive four-space in $\Mukai(S,\br)$. Similarly, one has $\OO^+(\Num(S))$.

\begin{theorem}[Huybrechts--Macrì--Stellari \cite{HMS09}]
The image of $\rho$ in Theorem \ref{02_K3-Orlov} is the subgroup $\OO^+(\MukaiZ{S})\subset \OO(\MukaiZ{S})$ and the restriction $\rho\colon\Aut(\DbS)\rightarrow\OO(\Num(S))$ has image contained in the subgroup $\OO^+(\Num(S))\subset\OO(\Num(S))$.
\end{theorem}

The behavior of an autoequivalence on $\DbS$ is also controlled by its action on the so-called spherical objects due to \cite[Appendix A]{Hu12}.

\begin{definition}
An object $E$ in $\DbS$ is \defi{spherical} if $\deriver\Hom(E,E)\cong\bc\oplus\bc[-2]$ and a Mukai vector $v$ is said to be \emph{spherical} if one has $v^2=-2$.
\end{definition}

A further study of the homomorphism $\rho$ has been carried out in \cite{Hu16}, using stability conditions on K3 surfaces.

\subsection{Stability conditions on K3 surfaces}
The notion of stability conditions is introduced in \cite{Bri07}, and stability conditions on K3 surfaces are intensively studied in \cite{BM14,BM14-MMP,Bri08}. Here we will adapt the formulation in \cite[Appendix 1]{BMS16} based on the locally-finite full stability conditions in the sense of \cite{Bri07,Bri08}.

\begin{definition}
Let $\schd$ be a triangulated category. A \defi{slicing} on $\schd$ is a collection of full subcategories $\schp(\phi)\subset\schd$ for all $\phi\in\br$, such that
\begin{itemize}
	\item $\schp(\phi+1)=\schp(\phi)[1]$ and $\Hom(\schp(\phi_1),\schp(\phi_2))=0$ for any $\phi_1>\phi_2$;
	\item every object $E$ of $\schd$ admits a Harder--Narasimhan filtration
\begin{displaymath}
	0= \xymatrix{
		E_0\ar[r]^{f_1}&E_1\ar[r]^{f_2}&E_2\ar[r]^{f_3}&\cdots\ar[r]^{f_{n-1}}&E_{n-1}\ar[r]^{f_n}&E_n
	} =E
\end{displaymath}
	with $0\neq A_i:=\Cone(f_i)$ in $\schp(\phi_i)$ for all $i$ such that $\phi_1>\cdots>\phi_n$.
\end{itemize}
The category $\schp(\phi)$ is quasi-abelian. The subcategory $\schp((0,1])\subset\schd$ generated by the subcategories $\schp(\phi)$ for $\phi\in(0,1]$ by extension is the heart of a bounded $t$-structure on $\schd$, and is called the \emph{heart} of this slicing.
\end{definition}

The Harder--Narasimhan filtration for an object $E$ in $\schd$ is unique up to a unique isomorphism. In this case, the objects $A_i$ are called the \emph{factors} of $E$. 

\begin{definition}
Let $\schd$ be a triangulated category and let $v\colon\Kum(\schd)\rightarrow \Lambda$ be a group homomorphism from the Grothendieck group of $\schd$ to some lattice $\Lambda$; a \defi{stability condition} on $\schd$ with respect to $v$ is a pair $\sigma=(\schp,Z)$ where $\schp$ is a slicing on $\schd$ and $Z$ is a linear map $\Lambda\rightarrow\bc$, called the \emph{central charge}, such that
\begin{itemize}
		\item one has $Z(v(E))\in\br_{>0}\cdot\exp(\pi\phi\bm{i})$ for any $0\neq E\in\schp(\phi)$;
	\item there exists a quadratic form $Q$ on the vector space $\Lambda_{\br}$ such that $\ker(Z_{\br})$ is negative definite with respect to $Q$, and one has $Q(v(E))\geq0$ for any $E\in\schp(\phi)$ and any $\phi\in\br$.
\end{itemize}
In this case, the quasi-abelian category $\schp(\phi)$ is indeed abelian. A non-zero object of $\schd$ is called \defi{semistable} of phase $\phi$ with respect to $\sigma$ once it belongs to $\schp(\phi)$ and is called \defi{stable} of phase $\phi$ with respect to $\sigma$ once it is simple in $\schp(\phi)$.
\end{definition}

The abelian category $\schp(\phi)$ has finite length (see e.g.~\cite[page 330]{Bri07}) for a stability condition $\sigma=(\schp,Z)$, so a non-zero object in $\schp(\phi)$ admits a Jordan--Hölder filtration with uniquely determined Jordan--Hölder factors. Two semistable objects are called \emph{$S$-equivalent} if they have the same Jordan--Hölder factors.

\begin{theorem}[Bridgeland \cite{Bri08}]\label{02_K3-Bridgeland}
Consider a K3 surface $S$, two classes $\omega$ and $\beta$ in $\NS(S)_{\br}$ with $\omega$ ample, the linear function $Z_{\omega,\beta}\colon \Num(S)\rightarrow\bc$ such that
$$Z_{\omega,\beta}(r,\Delta,s)= \frac{1}{2}(2\beta\cdot\Delta-2s+r(\omega^2-\beta^2))+(\Delta-r\beta)\cdot\omega\bm{i}$$
and the abelian category $\scha_{\omega,\beta}$ defined by tilting $\cate{Coh}(S)$ at $\beta\cdot\omega$ with respect to the slope function $\mu_{\omega}$ associated with $\omega$. Then there exists a stability condition $\sigma_{\omega,\beta}$ with respect to the Mukai vector $v\colon\Kum(S)\rightarrow\Num(S)$ having central charge $Z_{\omega,\beta}$ and heart $\scha_{\omega,\beta}$ if and only if $Z_{\omega,\beta}(v(\shf))\notin\br_{\leq0}$ for any spherical sheaf $\shf$ on $S$.
\end{theorem}

Due to \cite{Bay19,Bri07}, the set $\Stab(S)$ of stability conditions on $\DbS$ with respect to the Mukai vector $v$ carries a complex manifold structure. It is also shown in \cite{Bri08} that $\sigma_{\omega,\beta}$ are in the same connected component in $\Stab(S)$, denoted by $\Stab^{\dag}(S)$.

\begin{remark}\label{02_K3-Bridgeland-remark}
The topology of $\Stab^{\dag}(S)$ is given by the local isomorphism
$$\Stab^{\dag}(S)\rightarrow\Hom(\Num(S),\bc),\quad\sigma=(\schp,Z)\mapsto Z$$
according to \cite{Bay19,Bri07}. The image of this morphism is denoted by $\mathfrak{P}_0^+(S)$.
\end{remark}

\subsection{Autoequivalences and stability conditions}
There are two mutually commutative group actions on $\Stab(S)$ according to \cite[Lemma 8.2]{Bri07}.

The group $\Aut(\DbS)$ acts on $\Stab(S)$ on the left, through 
$$\Phi. (\schp, Z):= (\schp', Z \circ \Phi_*^{-1})  $$
where $\Phi_*=\Phi^{\coho}$ is the isometry on $\Num(S)$ and $\schp'$ is defined by $\schp'(\phi):=\Phi(\schp(\phi))$.

The universal cover $\grp$ of the group $\GLp(2,\br)$ acts on the right,
 via 
$$
(\schp, Z) . (M, f):= (\schp(f(\phi)), M^{-1} \circ Z). 
$$ 
where we use $\bc=\br\oplus\br\bm{i}$ to validate the composite $M^{-1}\circ Z\colon \Num(S)\rightarrow\bc$.

The $\grp$-action preserves $\Stab^{\dag}(S)$, while we only know that $\Stab^{\dag}(S)$ is invariant under the $\Aut(\DbS)$-action if $S$ has Picard rank one \cite{BB17}. On the other hand, the following statement always holds.

\begin{theorem}[Huybrechts \cite{Hu16}]\label{02_K3-Hu-conway}
Consider a K3 surface $S$ and a stability condition $\sigma=(\schp,Z)$ in $\Stab^{\dag}(S)$, then the following groups
\begin{align*}
	\Aut(\DbS,\sigma)&:=\{\Phi\in\Aut(\DbS)\,|\,\Phi.\sigma=\sigma\}\\
	\OO^+(\MukaiZ{S},Z)&:=\{\phi\in\OO^+(\MukaiZ{S})\,|\,Z\circ\phi=Z\}
\end{align*}
are finite, and $\rho$ in Theorem \ref{02_K3-Orlov} induces an isomorphism between them.
\end{theorem}

\begin{remark}
This criterion can be generalized to the relative setting and one can find it in the upcoming work \cite{BPPZ26}.
\end{remark}

Moreover, through \cite{BB17} and a careful study of $\Stab^{\dag}(S)$, a better understanding of the group $\Aut(\DbS,\sigma)$ has been obtained for a general K3 surface.

\begin{theorem}[Fan--Lai \cite{FL23}]\label{02_FL-result}
Consider a K3 surface $S$ with Picard rank one.
\begin{enumerate}
	\item The group $\Aut(\DbS,\sigma)$ is either trivial or isomorphic to the cyclic group $\Stwo$ for a stability condition $\sigma\in\Stab^{\dag}(S)$;
	\item Any non-trivial involution in $\Aut(\DbS)$ is non-symplectic and fixes a (not necessarily unique) stability condition in $\Stab^{\dag}(S)$;
	\item\label{02_FL-result(3)} Suppose that the degree of $S$ is $2t$, then the number of conjugacy classes of the non-trivial involutions on $\DbS$ is equal to $1$ if $t=1,2$; is equal to $0$ if $t$ is divisible by $4$; and is equal to 
	$$
		\frac{1}{2}\prod_{ p|t}\left(1+(-1)^{\frac{p-1}{2}}\right)
	$$
	in any other case, where $p$ runs over the odd prime factors of $t$.
\end{enumerate}
\end{theorem}

This theorem will be revisited later in Section \ref{000_Section-2.2} in order to formulate the notion of Bridgeland--Enriques general K3 surfaces.

\begin{example}
Let $S$ be a K3 surface with Picard rank one and degree $2$, then one has $\Aut(S)=\bz/2\bz$ and the unique conjugacy class of non-trivial involutions on $\DbS$ is represented by $\tau_*$ for the unique non-trivial $\tau\in\Aut(S)$.
\end{example}

\begin{example}
Let $S$ be a K3 surface with Picard rank one and degree $4$, then the unique conjugacy class of non-trivial involutions on $\DbS$ is represented by the involution described in \cite[Section 8.1]{KP17} (see also \cite[Example 3.17]{Liu26b}).
\end{example}

\begin{example}
Let $S$ be a K3 surface with Picard rank one and degree $10$, then the unique conjugacy class of non-trivial involutions on $\DbS$ is represented by the involution described in \cite[Section 8.2]{KP17} (see also \cite[Example 3.18]{Liu26b}).
\end{example}

\subsection{Enriques categories over K3 surfaces}
The notion of an Enriques category is introduced in \cite{BP23}. Here let us recall some basic definitions.

\begin{definition}
An autoequivalence $\funct{S}_{\scht}$ on a triangulated category $\scht$ is called a \defi{Serre functor} if one has a functorial isomorphism
$$\Hom(A,B)\cong\Hom(B,\funct{S}_{\scht}A)^\vee$$
for any objects $A$ and $B$ in $\scht$.
\end{definition}

The Serre functor is unique, once it exists. An admissible subcategory of the bounded derived category of a smooth projective variety admits a Serre functor.

\begin{definition}
A \defi{2-Calabi--Yau category} $\schd$ is a triangulated category which has a Serre functor isomorphic to the shift functor $[2]$ on $\schd$.
\end{definition}

The bounded derived category of a K3 surface or an abelian surface is by definition a 2-Calabi--Yau category.

\begin{definition}
An \defi{Enriques category} $\sche$ is the equivariant category $\schd_{\Stwo}$ for an $\Stwo$-action on a $2$-Calabi--Yau category $\schd$ such that the Serre functor of $\sche$ has the form $\amalg\circ[2]$ where $\amalg$ is a generator of the residual $\mathfrak{S}^\vee_2$-action on $\sche$.
\end{definition}

The prototypes of Enriques categories are bounded derived categories of Enriques surfaces. The first examples of Enriques categories over K3 surfaces with Picard rank one are due to \cite{KP17}. In general, one has

\begin{proposition}\label{02_Enriques-over-K3-cri}
Consider an $\Stwo$-action on $\DbS$ for a K3 surface $S$ generated by an
involution $\Pi$, then the equivariant category $\sche=\DbS_{\Stwo}$ is an Enriques category if and only if the involution $\Pi$ is non-symplectic.

\begin{proof}
It is a combination of \cite[Lemma 6.5]{BO23} and \cite[Proposition 1.11]{Liu26a}.
\end{proof}
\end{proposition}

Also, one has the following criterion deduced from \cite[Corollary 4.11]{BO23}.

\begin{proposition}\label{02_Enriques-group-action-cri}
An involution $\Pi$ on $\DbS$ induces an $\Stwo$-group action on $\DbS$ if there exists an object $E$ in $\DbS$ fixed by $\Pi$ satisfying $\Hom(E,E)=\bc$.
\end{proposition}

By the construction of equivariant categories, two Enriques categories over the same K3 surface $S$ are equivalent if and only if the involutions generating the two $\Stwo$-actions on $\DbS$ are conjugate to each other.

\subsection{Moduli space of semistable objects}
An admissible subcategory $\scht$ inside $\cate{D}^b(X)$ of a smooth projective variety $X$ admits a finite-dimensional $\Num(\scht)$. Let $\sigma$ be a stability condition on $\scht$ and $v$ be a class in $\Num(\scht)$, one can define a moduli stack $\mathfrak{M}_{\sigma}(v)$ of $\sigma$-semistable objects with class $v$. A comprehensive reference is \cite{BLMNPS21}.

\begin{definition}
The moduli stack $\mathfrak{M}_{\sigma}(v)\colon(\cate{Sch}/\bc)^{\cate{op}}\rightarrow\cate{Gpds}$ is defined by
$$
\mathfrak{M}_{\sigma}(v)(T)= \left \{ E\in\cate{D}^b(X_T) ~  \middle\vert ~ 
\begin{array}{l} 
	\,E|_{X_t}\textup{ is in }\scht_t\subset\cate{D}^b(X_t)\textup{ and is} \\ 
	\sigma\textup{-semistable of class }v\textup{ and with} \\
	\textup{phase in }(-1,1]\textup{ for any }t\in T 
\end{array}  \right \}
$$
for any locally finitely presented complex scheme $T$. Similarly, one defines the moduli substack $\mathfrak{M}^{\stable}_{\sigma}(v)$ of $\sigma$-stable objects with class $v$.
\end{definition}

Given a class $v\in\Num(S)$ and a $v$-generic stability condition $\sigma$ in $\Stab^{\dag}(S)$ for a K3 surface $S$, then \cite{BM14} asserts that the moduli stack $\mathfrak{M}_{\sigma}(v)$ admits a projective coarse moduli space $\modulis_{\sigma}(v)$ whose points are $S$-equivalent classes of $\sigma$-semistable objects in $\schp(-1,1]$ with Mukai vector $v$. Thanks to the support property, one has

\begin{proposition}[{{\cite[Section 9]{Bri08} and \cite[Proposition 2.3]{BM14}}}]\label{wall-and-chamber}
Given $v\in\Num(S)$, then there exists a locally finite set of walls (real codimension one submanifolds with boundary) in the complex manifold $\Stab^{\dag}(S)$, depending only on $v$, such that:
	\begin{itemize}
		\item The walls split $\Stab^{\dag}(S)$ into open and path-connected chambers. 
		\item The set of semistable and stable objects with class $v$ does not change for stability conditions in a chamber;
		\item When $\sigma$ lies on a single wall, then there is a $\sigma$-semistable object that is unstable in one adjacent chamber, and semistable in the other.
		\item When we restrict to an intersection of finitely many walls $W_1\cap\cdots\cap W_k$, we obtain a wall-and-chamber
		decomposition on $W_1\cap\cdots\cap W_k$ with the same properties, where the walls are given by $W \cap W_1\cap\cdots\cap W_k$.
	\end{itemize}
If $v$ is primitive in $\Num(S)$, then a stability condition $\sigma$ lies on a wall with respect to $v$ if and only if there exists a strictly $\sigma$-semistable object of class $v$. 
\end{proposition}

The behavior of $\modulis_{\sigma}(v)$ with $\sigma$ moving across a wall is studied intensively in \cite{BM14,BM14-MMP,MZ16}. Here we only recall some necessary definitions for this article.

Suppose that $W\subset\Stab^{\dag}(S)$ is a wall with respect to a given $v\in \Num(S)$ with $\modulis_{\sigma_0}(v)\neq\varnothing$ and $\sigma_0\in W$ is not in any other wall, then $W$ is called a \emph{totally semistable wall} if the stable locus $\modulis^{\stable}_{\sigma_0}(v)$ is empty.

Suppose in addition that $v^2>0$ and $\sigma_{\pm}$ are $v$-generic stability conditions in different adjacent chambers of $W$. Then \cite[Proposition 5.2]{MZ16} provides a birational morphism $\pi_{\pm}\colon \modulis_{\sigma_{\pm}}(v)\rightarrow\overline{\modulis}_{\pm}$ which contracts curves in $\modulis_{\sigma_{\pm}}(v)$ parameterizing $S$-equivalent objects under $\sigma_0$, where $\overline{\modulis}_{\pm}$ is the image of $\pi_{\pm}$ in $\modulis_{\sigma_0}(v)$. Then the wall $W$ is called a \emph{fake wall} if $\pi_{\pm}$ are isomorphisms, a \emph{flopping wall} if $\pi_{\pm}$ are small contractions, and a \emph{divisorial wall} if $\pi_{\pm}$ are divisorial contractions.

\begin{definition}
A stability condition $\sigma$ is called $v$\defi{-generic} for a given $v\in\Num(S)$ if it is not contained in any wall with respect to $v$.
\end{definition}

The following two theorems are established in \cite{BM14}.

\begin{theorem}\label{11_bridgeland-hyperk-K3}
Consider a primitive Mukai vector $v\in\Num(S)$ with $v^2=0$ and a $v$-generic stability condition $\sigma$ in $\Stab^{\dag}(S)$, then $\modulis_{\sigma}(v)$ is a K3 surface and there exists a Hodge isometry
$$\theta_{v,\sigma}\colon v^{\perp}/\bz v\stackrel{\sim}{\rightarrow} \HH^2(\modulis_{\sigma}(v),\bz)$$	
where $v^{\perp}$ and $\bz v$ are sublattices in $\MukaiZ{S}$.
\end{theorem}

\begin{theorem}\label{11_bridgeland-hyperk-K3[n]}
Consider a primitive Mukai vector $v\in\Num(S)$ with $v^2=2n-2\geq2$ and a $v$-generic stability condition $\sigma$ in $\Stab^{\dag}(S)$, then $\modulis_{\sigma}(v)$ is a $\KK^{[n]}$-type hyperkähler manifold and there exists a Hodge isometry
$$\theta_{v,\sigma}\colon v^{\perp}\stackrel{\sim}{\rightarrow} \HH^2(\modulis_{\sigma}(v),\bz)$$
where $v^{\perp}$ is the sublattice in $\MukaiZ{S}$.
\end{theorem}

Moreover, the stability conditions on K3 surfaces can be defined in families and one has the relative moduli space of semistable objects similar to the relative moduli space of semistable sheaves \cite[Theorem 4.3.7]{HL} according to \cite{BLMNPS21}. We will not recall the explicit definitions as this article does not essentially use them.

\section{Bridgeland--Enriques General K3 Surfaces}\label{000_Section-11}

In this section, we will study the locus of Bridgeland--Enriques general K3 surfaces. In particular, we will prove Theorem \ref{00_main-result-1} and Theorem \ref{00_main-result-1.2}.

\subsection{The first example of a Bridgeland--Enriques general locus}
Here let us first explain Example \ref{00_example-t=1} as a warm-up.

\begin{proposition}\label{11_first-example}
Any polarized K3 surface $(S,H)$ of degree $2$ is Bridgeland--Enriques general. Moreover, one has $\MBE(2)=\MKK(2)$.

\begin{proof}
The ample line bundle $H$ realizes $S$ as a double cover $S\rightarrow\bp^2$ branched over a degree six curve. The covering involution $f$ on $S$ is known to be non-symplectic and induces an $\Stwo$-action on $\DbS$. The equivariant category $\DbS_{\Stwo}$ is Enriques due to Proposition \ref{02_Enriques-over-K3-cri}. Therefore, one sees directly that the Hodge isometry 
$$(r,\Delta,s)\mapsto(r,(\Delta\cdot H)H-\Delta,s)$$
on $\MukaiZ{S}$ induced by the involution $f$ makes $\MBE(2)=\MKK(2)$.
\end{proof}
\end{proposition}

Then we would like to mention that there can be more than one group action on a Bridgeland--Enriques general K3 surface.

\begin{example}[{{\cite[Example 3.11]{PPZ26}}}]\label{11_t=2-example1}
A smooth section $S\in |-K_W|$ of a bidegree $(1,1)$-divisor $W$ in $\bp^2\times\bp^2$ is a K3 surface. The two rulings induce nef line bundles $H_1$ and $H_2$ on the K3 surface $S$ such that $H_i^2=2$ and $H_1\cdot H_2=4$.

A \emph{special Verra threefold} is the double cover $V\rightarrow W$ branched over a smooth section $S\in |-K_W|$. One defines $\schk^i_V$ for $i=1,2$ by $\cate{D}^b(V)=\langle\schk^i_V,\pi^*_i\cate{D}^b(\bp^2)\rangle$ where $\pi_i\colon V\rightarrow\bp^2$ is the projection to the $i$-th $\bp^2$ factor. The covering involution on $V$ induces an $\Stwo$-action on $\schk_V^i$ such that $(\schk_V^i)_{\Stwo}\cong\DbS$ by \cite{KP17}.

The residual action on $\DbS$ is generated by the involution
$$\Pi_i:=\funct{T}_{\sho_S}\funct{T}_{\sho_S(H_i)}\funct{T}_{\sho_S(2H_i)}(-\otimes\sho(H_1+H_2))[-1]$$
according to \cite[Proposition 7.10]{KP17}. The Hodge isometry $\tau_1$ induced by $\Pi_1$ is 
$$(r,\Delta,s)\mapsto(r,(2r+a)H_1-rH_2-\Delta,s-3r-2a+b)$$
where $a:=\Delta\cdot H_1$ and $b:=\Delta\cdot H_2$. Up to a sign, the only primitive vector $v$ in 
$$\Omega_S:=\HH^0(S,\bz)\oplus \bz[H_1]\oplus\bz[H_2]\oplus \HH^4(S,\bz)$$
satisfying $\tau_1(v)=-v$ is $v=(0,-2H_1+H_2,3)$ which has divisibility $3$ in $\Omega_S$.

A general $S$ in this example has Picard rank two and $H_1$ is an ample class with square $2$ which induces a covering involution on $S$. One checks that, up to a sign, the only primitive vector $v'$ in $\Num(S)$ satisfying $f^*v'=-v'$ is $v'=(0,-2H_1+H_2,0)$ which has divisibility $6$ in $\Num(S)=\Omega_S$. It means that $\Pi_1$ is not conjugate to the direct image $f_*$ of the covering involution coming from $H_1$.
\end{example}

\begin{remark}
In fact, one has $\cate{D}^b(S)\cong\cate{D}^b(S',\alpha')$ for a degree $2$ K3 surface $S'$ and a 2-torsion Brauer class $\alpha'$ on $S'$ (see \cite[Example 3.23]{Liu26b} and \cite[Remark A.12]{Liu26a}).
\end{remark}

\begin{example}[{{\cite[Section 6.2]{JM26}}}]\label{11_t=2-example2}
A smooth section $S\in |-K_W|$ of the blow-up $W:=\Bl_p(\bp^3)$ is a K3 surface and it has two divisor classes $H$ and $E$ such that $H$ comes from $\sho_{\bp^3}(1)$ and $E$ is the restriction to $S$ of the exceptional divisor. One sees that $H^2=4,E^2=-2$ and $E\cdot H=0$. The double cover $X\rightarrow W$ branched over $S$ can be associated with a category $\schk_X$ such that the covering involution on $X$ induces an $\Stwo$-action on $\schk_X$ satisfying $(\schk_X)_{\Stwo}\cong\DbS$ by \cite{KP17}. 

The residual action on $\DbS$ is generated by
$$\Pi:=\funct{T}_{\sho_S(-H)}\funct{T}_{\sho_S(-E)}\funct{T}_{\sho_S}(-\otimes\sho(2H-E))[-1]$$
according to \cite{JM26,KP17} and the associated Hodge isometry $\tau$ is
$$(r,\Delta,s)\mapsto(-9r-3a-2s,-bE+(15r+5a+3s)H-\Delta,-50r-15a-9s)$$
where $a:=\Delta\cdot H$ and $b:=\Delta\cdot E$. Up to a sign, the only primitive vector $v$ in 
$$\Omega_S:=\HH^0(S,\bz)\oplus \bz[H]\oplus\bz[E]\oplus \HH^4(S,\bz)$$
satisfying $\tau(v)=-v$ is $v=(2,-3H,10)$ which has divisibility $2$ in $\Omega_S$.

A general $S$ in this example has Picard rank two and $D:=H-E$ is an ample class with square $2$ which induces a covering involution on $S$. One checks that, up to a sign, the only primitive vector $v'$ in $\Num(S)$ satisfying $f^*v'=-v'$ is $v'=(0,H-2E,0)$ which has divisibility $4$ in $\Num(S)=\Omega_S$. It means that $\Pi$ is not conjugate to the direct image $f_*$ of the covering involution coming from $D$.
\end{example}

\subsection{The conjugacy classes of involutions}\label{000_Section-2.2}
In order to understand the locus $\MBE(2t)$ for $t\geq 2$, we would like to reformulate Theorem \ref{02_FL-result} in terms of modular curves. The reader is referred to \cite{DS} for basic definitions.

Suppose that $S$ has Picard rank one and degree $2t$, then any $\Phi$ in $\Aut(\DbS)$ induces an isometry $\phi$ on the upper half plane $\bh$ by \cite[Section 2.1]{Ka14}. The isometry $\phi$ fixes a point in $\bh$ when $\Phi$ is a non-trivial involution. So the non-trivial involutions on $\DbS$ correspond to points in $\bh$. The Hecke congruence subgroup
$$\Gamma_0(t):=\left\{\left.\begin{pmatrix}
	a_{11}&a_{12}\\a_{21}&a_{22}
\end{pmatrix}\in\PSL(2,\bz)\right|a_{21}\equiv0\pmod{t}\right\}$$
acts on $\bh$ and the quotient $\Gamma_0(t)\git\bh$ is the classical modular curve. 

\begin{theorem}[{{\cite[Theorem 5.6]{FL23}}}]
Consider a Picard rank one K3 surface with degree $2t$, then the order-two elliptic points in $\Gamma_0(t)\git\bh$ and the conjugacy classes of non-trivial involutions on $\DbS$ have the following relations.
\begin{enumerate}
	\item If $t=2$, then the unique conjugacy class of non-trivial involutions on $\DbS$ corresponds to the unique order-two elliptic point in $\Gamma_0(2)\git\bh$;
	\item If $t\geq 5$, then a conjugacy class of non-trivial involutions on $\DbS$ corresponds to two distinct order-two elliptic points in $\Gamma_0(t)\git\bh$. 
\end{enumerate}
\end{theorem}

One notices that a point in $\Gamma_0(t)\git\bh$ is a $\Gamma_0(t)$-orbit in $\bh$, so a conjugacy class also corresponds to one or two $\Gamma_0(t)$-orbits. In particular, two non-trivial involutions on $\DbS$ are conjugate to each other if and only if the points in $\mathbb{H}$ fixed by them are contained in the corresponding $\Gamma_0(t)$-orbits.

\begin{example}\label{11_modular-curve-t=2}
Let $S$ be a K3 surface with Picard rank one and degree $4$, then the unique conjugacy class of involutions corresponds to the orbit
$$\Gamma_0(2)\frac{-1+\bm{i}}{2}$$
in $\bh$. One should compare this with \cite[Proposition 5.2]{Liu26a}.
\end{example}

\begin{example}
Let $S$ be a K3 surface with Picard rank one and degree $10$, then the unique conjugacy class of involutions corresponds to the two orbits
$$\Gamma_0(5)\frac{-2+\bm{i}}{5}\quad\textup{and}\quad \Gamma_0(5)\frac{-3+\bm{i}}{5}$$
in $\bh$. One should compare this with \cite[Proposition 6.2]{Liu26a}.
\end{example}
 
In general, the subsequent fact is known and one can show it by following \cite[Exercise 3.7.6]{DS} and noticing that $\frac{a+\bm{i}}{a^2+1}$ and $\frac{a+\bm{i}}{t}$ are in the same $\Gamma_0(t)$-orbit.

\begin{lemma}\label{11_modular-curve-lemma}
Consider $t\geq 2$, then the order-two elliptic points in $\Gamma_0(t)\git\mathbb{H}$ are
$$\Gamma_0(t)\frac{-a+\bm{i}}{t}$$
where $a\in[1,t]$ is an integer such that $a^2+1\equiv0\pmod{t}$. 

\begin{proof}
Here we only sketch an argument that $\frac{a+\bm{i}}{a^2+1}$ and $\frac{a+\bm{i}}{t}$ are in the same $\Gamma_0(t)$-orbit. One takes the integer $n\geq 2$ such that $a^2+1=(n-1)t$, then using the unique factorization in $\bz[\bm{i}]$ one can find a pair of integers $(p,q)$ such that
$$p^2+q^2=n-1$$
and $-q+ap=u(n-1)$ for some integer $u$ (see later in Proposition \ref{11_BE-locus-square-2-vector}). 

Then one can see that 
$$\begin{pmatrix}
ap+q&-p\\
pt&-u
\end{pmatrix}\frac{a+\bm{i}}{a^2+1}=\frac{a+\bm{i}}{t}$$
in $\bh$ and therefore conclude.
\end{proof}
\end{lemma}

It allows us to make the following observation.

\begin{proposition}\label{11_involution-description}
Consider a K3 surface $S$ with Picard rank one and degree $2t$ for some $t\geq 2$, then one can find a representative $\Pi$ for any conjugacy class of non-trivial involutions in $\Aut(\DbS)$ such that 
\begin{enumerate}
	\item\label{11_involution-description(1)} the Hodge isometry $\phi$ on $\MukaiZ{S}$ induced by $\Pi$ preserves the vectors 
	\begin{center}
	$v_1=(1,0,1-n)$ and $v_2=(a,-(n-1)H,(n-1)a)$
	\end{center}
	for some $n\geq2$ and acts on the sublattice $v_1^{\perp}\cap v_2^{\perp}\subset\MukaiZ{S}$ as $v\mapsto -v$;
	\item the Hodge isometry $\bar{\phi}$ on $\Num(S)$ induced by $\Pi$ preserves the function
	$$Z_{t,a}(r,\Delta,s)=-\left(\frac{a^2-1}{t}r+\frac{a}{t}\Delta\cdot H+s\right)+\left(\frac{2ar+\Delta\cdot H}{t}\right)\bm{i}$$
    where $H$ is the degree $2t$ ample line bundle on $S$.
\end{enumerate}

\begin{proof}
A solution $a\in[1,t]$ such that $a^2+1\equiv0\pmod{t}$ will give a minimal positive solution $(a,1)$ for the negative Pell's equation
$$
X^2-t(n-1)Y^2=-1
$$
where $n:=\frac{a^2+1}{t}+1\geq 2$. By virtue of \cite[Section 3]{Liu26b}, it gives a non-trivial Hodge involution $\bar{\phi}$ on $\Num(S)$ whose action on $\bh$ can be described as
$$z\mapsto-\frac{az+(n-1)}{tz+a}$$
and whose fixed point is exactly $\frac{-a+\bm{i}}{t}$. In this case, the map $\phi$ described in (1) is indeed a Hodge isometry in $\OO^+(\MukaiZ{S})$ and restricts to $\bar{\phi}$. 

It is checked in \cite[Proposition 3.7]{Liu26b} that (2) is true, and $Z_{t,a}$ is the central charge of a stability condition 
$$\sigma_{t,a}:=\sigma_{\frac{1}{t}H,-\frac{a}{t}H}$$
in $\Stab^{\dag}(S)$ according to \cite[Proposition 3.8]{Liu26b}. So one can lift $\phi$ to a non-trivial involution $\Pi$ on $\DbS$ using Theorem \ref{02_K3-Hu-conway} (see also \cite[Proposition 3.16]{Liu26b}). Then one concludes by Lemma \ref{11_modular-curve-lemma} for $t\geq 5$ and by Example \ref{11_modular-curve-t=2} for $t=2$.
\end{proof}
\end{proposition}

\begin{remark}
One notices that $(t,-aH,t(n-1))$ is always contained in $v_1^{\perp}\cap v_2^{\perp}$ and it is a vector of square $-2t$.
\end{remark}

In particular, one has a finite set $\HR(2t)$ for every $t\geq 2$ which contains the Hodge isometries in Proposition \ref{11_involution-description} \eqref{11_involution-description(1)}. Moreover, one has

\begin{proposition}\label{11_BE-locus-square-2-vector}
Consider integers $n,t\geq2$ and $a>0$ such that
$$a^2-t(n-1)=-1$$
then the invariant lattice of the Hodge isometry $\phi$ in Proposition \ref{11_involution-description} \eqref{11_involution-description(1)} is spanned by two orthogonal square 2 vectors 
\begin{displaymath}
		\bar{v}_1=(\frac{p+aq}{n-1},-qH,-p+aq)\quad\textup{and}\quad\bar{v}_2=(\frac{q-ap}{n-1},pH,-q-ap)
\end{displaymath}
where $(p,q)$ is an integer pair such that $p^2+q^2=n-1$ and $p+aq\equiv0\pmod{n-1}$.
	
\begin{proof}
It is sufficient to find a pair $(p,q)$ of integers satisfying the aforementioned conditions. One claims that there exists a factorization
\begin{equation}\label{11_BE-locus-square-2-vector-factor}\tag{$\star$}
a+\bm{i}=(u+v\bm{i})(p-q\bm{i})
\end{equation}
in the unique factorization domain $\bz[\bm{i}]$, such that $u^2+v^2=t$. In this case, one has $p^2+q^2=n-1$. Also, one can see $up+vq=a$ and $vp-uq=1$ from the factorization of $a+\bm{i}$, so that $p+aq=v(p^2+q^2)=v(n-1)$. Similarly, one has
$$-q+ap=-q+up^2+vpq=u(p^2+q^2)=u(n-1)$$
thus $\bar{v}_2$ is also an integral vector. 
		
It remains to show the claim \eqref{11_BE-locus-square-2-vector-factor}. A prime factor of $t$ either has the form $4k+1$ or is $2$. One has a unique factorization $4k+1=\bm{z}\cdot\bar{\bm{z}}$ in $\bz[\bm{i}]$ such that exactly one of $\bm{z}$ or $\bar{\bm{z}}$ divides $a+\bm{i}$. If $t$ has a factor $2$, then $a$ is odd and one has a unique factorization $2=(1+\bm{i})(1-\bm{i})$. So there is a factor of $a+\bm{i}$ with norm $t$.
\end{proof}
\end{proposition}

\begin{remark}
It means that the invariant lattice is always isomorphic to $\Ala_1^2$ and there are no isotropic (i.e.~square zero) or spherical vectors in it.
\end{remark}

\begin{remark}
Moreover, the Hodge isometries in $\HR(2t)$ can be defined on the Mukai lattice $\MukaiZ{S}$ for any polarized K3 surface $(S,H)$ in $\MKK(2t)$.
\end{remark}

One can fix the integers $p,q,u,v$ in the factorization \eqref{11_BE-locus-square-2-vector-factor} for any valid $(t,a)$ by assuming that $p>0$ and $q$ is non-negative. In this case, ranks of the Mukai vectors $\bar{v}_1$ and $-\bar{v}_2$ are positive when $t\geq 2$ by construction.

Henceforth, we will fix the choices of integers $p,q,u,v$ for any pair $(t,a)$ such that $\MBE(2t)_a$ is non-empty.

\begin{corollary}\label{11_BE-central-charge}
Suppose that $\MBE(2t)_a$ is non-empty for some $t\geq 2$, then
$$Z_{t,a}(r,\Delta,s)=\frac{1}{t}(v+u\sqrt{-1})(e_1+e_2\sqrt{-1})$$
where $e_i:=\langle(r,\Delta,s), \bar{v}_i\rangle$.
\end{corollary}

The following observation is also interesting. Here the K3 surface $S$ can have arbitrary Picard rank but we still need $t\geq 2$.

\begin{proposition}\label{11_BE-large-volume-limit}
Consider a polarized K3 surface $(S,H)$ in $\MBE(2t)_a$ and let $w$ be either $\bar{v}_1$ or $-\bar{v}_2$ as above, then the following statements are equivalent:
\begin{itemize}
	\item[(a)] there exists a strictly semistable sheaf in $\modulis_H(w)$;
	\item[(b)] the rank of $w$ is even and there exists a divisor $D$ on $S$ such that $D\cdot H=t$ and $2D^2=t-5$.
\end{itemize}
In this case, the strictly semistable sheaf has two spherical stable factors.

\begin{proof}
Here we only treat the case $w=\bar{v}_1$ and the other case is similar. Suppose that an element $\shg$ in $\modulis_H(w)$ admits a non-trivial Jordan--Hölder factor, say with Mukai vector $w_0=(r_0,\Delta,s_0)$. Then one has $0<r_0<r$ and
$$\frac{\Delta\cdot H}{r_0}=-\frac{2qt}{r}$$
where $r$ is the rank of $w$. One notices that $r$ is coprime to $qt$ by construction, so $r$ is a factor of $2r_0$. It follows that $r=2r_0$. Then $q$ is odd. So the class
$$D:=\Delta+\frac{q+1}{2}H$$
is integral and it remains to show that $2\Delta^2=tq^2-5$.

The identity $r=2r_0$ ensures that the semistable sheaf $\shg$ has exactly two stable factors with Mukai vectors $w_0$ and $w-w_0$ respectively. One has
$$w_0^2=(w-w_0)^2=\frac{2+E^2}{4}$$
where $E:=2\Delta+qH$. Since $q$ is odd, the fact $E\cdot H=0$ and the Hodge index theorem imply that $E^2<0$. So $w_0^2$ can only be $0$ or $-2$ by Mukai lemma. The case $w_0^2=0$ or equivalently $E^2=-2$ is impossible as $E^2=-2$ means that either $E$ or $-E$ is effective. Hence one has $w_0^2=-2$ and $E^2=-10$. So $2\Delta^2=tq^2-5$.

In conclusion, we finish (a) $\Rightarrow$ (b). The converse implication (b) $\Rightarrow$ (a) holds because of Yoshioka's nonemptiness theorem, see e.g. \cite[Theorem 5.2]{BM14}.
\end{proof}
\end{proposition}

\begin{corollary}
At least one of the moduli spaces $\modulis_H(\bar{v}_1)$ and $\modulis_H(-\bar{v}_2)$ contains only stable sheaves and therefore is a $\KK^{[2]}$-type hyperkähler manifold. Both of them only contain stable sheaves if the Picard rank of $S$ is one.
\end{corollary}

Special cases have already been checked in \cite{Liu26a,Liu26b,LZ26} and we summarize some of them in the following example.

\begin{example}\label{11_BE-example}
Suppose that $n=2$ and $t$ is a prime in Proposition \ref{11_BE-locus-square-2-vector}, then one can see $\bar{v}_1=(1,0,-1)$ and $\bar{v}_2=(-a,H,-a)$. Then $\modulis_H(\bar{v}_1)\cong S^{[2]}$. On the other hand, the moduli space $\modulis_H(-\bar{v}_2)$ contains a strictly semistable vector bundle
$$\sho_S(-E)\oplus\sho_S(E-H)$$
for a K3 surface $(S,H)$ of degree $10$ and a class $E$ satisfying $E^2=0$ and $E\cdot H=5$.
\end{example}

Now we are prepared to reach our first main theorem.

\subsection{The Bridgeland--Enriques general loci}
Here we are going to describe the Bridgeland--Enriques general locus $\MBE(2t)$ for $t\geq 2$. 

One recalls that the \emph{Noether--Lefschetz divisor} $\NL^{2t}_{n,2d}$ is a closed subset in the moduli space $\MKK(2t)$ of polarized K3 surfaces of degree $2t$ containing any K3 surface $(S,H)$ which has a divisor $D$ with $D^2=2d$ and $D\cdot H=n$.

\begin{proposition}\label{11_main-result-1-lemma1}
Consider a polarized K3 surface $(S,H)$ of degree $2t$ (with $t\geq 2$) and an integer $a\in [1,t]$ such that 
$$a^2+1=mt$$ 
for some positive integer $m$, then $(\omega,\beta)=(\frac{1}{t}H,-\frac{a}{t}H)$ gives a stability condition $\sigma_{\omega,\beta}$ as in Theorem \ref{02_K3-Bridgeland} if and only if $(S,H)$ is contained in the open locus
$$\MBE(2t)_a:=\MKK(2t)-\bigcup_{(n,d)\in I_{t,a}}\NL_{n,2d}^{2t}$$
where $I_{t,a}=\{(2ar,mr^2-1)\mid r\in\bz_+\textup{ and }r^2< t\}$ is a finite set of integer pairs.

\begin{proof}
According to Theorem \ref{02_K3-Bridgeland}, the stability condition $\sigma_{t,a}:=\sigma_{H/t,-aH/t}$ is well-defined if and only if $Z_{t,a}(v(\shf))\notin\br_{\leq0}$ for any spherical sheaf $\shf$ on $S$. It suffices to consider spherical sheaves with positive rank by \cite[Lemma 6.2]{Bri08}.

Suppose otherwise, then one has a class $(r,\Delta, s)$ in $\Num(S)$ such that
$$r > 0,\quad \Delta^2 = 2rs - 2,\quad \Delta \cdot H = -2ar,\quad (a^2-1)r + a\Delta \cdot H + ts \geq 0$$
at the same time. Set $d:=\Delta^2/2$, it directly follows that
$$(a^2+1)r^2-t(d+1)\leq0$$
and the Hodge index theorem implies $a^2r^2\geq td$. Then one has $r^2\leq t$. The equality $r^2=t$ cannot hold, otherwise the Hodge index theorem would force $r\Delta=-aH$ which is impossible as $H$ is primitive and $\gcd(r,a)=\gcd(t,a)=1$.

Also, one can see $d\geq mr^2-1$ using $s\geq mr$. It ensures that $d=mr^2-1$ as $d$ is an integer and one has $mr^2t-r^2=a^2r^2\geq td$.
\end{proof}
\end{proposition}

In particular, we prove that the function $Z_{t,a}$ is the central charge of some stability condition in $\Stab^{\dag}(S)$ if and only if $(S,H)$ is in $\MBE(2t)_a$. 

\begin{proposition}\label{11_main-result-1-lemma2}
Consider some $(S,H)$ in $\MBE(2t)_a$ with $t\geq 2$ and the stability condition $\sigma_{t,a}:=\sigma_{H/t,-aH/t}$, then the moduli space $\modulis_{\sigma_{t,a}}(\bar{v}_1)$ is either a hyperkähler manifold of $\KK^{[2]}$-type with a degree $2$ polarization or the flop contraction of a $\KK^{[2]}$-type hyperkähler manifold $\tilde{\modulis}_{t,a}$.

\begin{proof}
The moduli space $\modulis_{\sigma_{t,a}}(\bar{v}_1)$ is a $\KK^{[2]}$-type hyperkähler manifold if it does not contain any strictly semistable object. In this case, the $\Stwo$-action on $\cate{D}^b(S)$ induces a non-symplectic involution $f$ on $\modulis_{\sigma_{t,a}}(\bar{v}_1)$. The $f^*$-invariant sublattice in $\HH^2(\modulis_{\sigma_{t,a}}(\bar{v}_1),\bz)$ is spanned by a square two class $\theta$ according to Proposition \ref{11_BE-locus-square-2-vector} and \cite[Proposition 3.5]{Ou18}. In particular, either $\theta$ or $-\theta$ is ample. 

Suppose that an element $E$ in $\modulis_{\sigma_{t,a}}(\bar{v}_1)$ is strictly semistable, then one can take a non-trivial Jordan--Hölder factor thereof with Mukai vector $w$.

According to Corollary \ref{11_BE-central-charge}, one must have
$$Z_{t,a}(\bar{v}_1-w)=Z_{t,a}(w)=\frac{1}{2}Z_{t,a}(\bar{v}_1)=\frac{v}{t}+\frac{u}{t}\sqrt{-1}$$
since $u,v$ are coprime and $\Re Z_{t,a}(w),\Im Z_{t,a}(w)\in\frac{1}{t}\bz$ by construction. It follows that the strictly semistable object $E$ has exactly two distinct Jordan--Hölder factors say $A$ and $B$ with Mukai vectors $w$ and $\bar{v}_1-w$ respectively, such that
$$w=\frac{\bar{v}_1+w_0}{2}$$
for some $w_0\in\ker Z_{t,a}\cap\Num(S)$ satisfying $w_0^2=4w^2-2$ again by Corollary \ref{11_BE-central-charge}. 

Since
$$w^2=(\bar{v}_1-w)^2$$
one must have $w^2\in\{0,-2\}$ by Mukai lemma.
One notices that $w$ cannot be an isotropic vector otherwise $w_0^2=-2$, contradicting the fact that $\sigma_{t,a}$ is a stability condition (see Theorem \ref{02_K3-Bridgeland}). It means that $A$ and $B$ are spherical. One notices that $w\cdot (\bar{v}_1-w)=3$ and it is not hard to conclude $\Ext^1(A,B)=\bc^3$.

In particular, the singular locus of $\modulis_{\sigma_{t,a}}(\bar{v}_1)$ is a finite subset and the points in it are represented by direct sums of two spherical stable objects. In a small enough neighborhood of $\sigma_{t,a}$ in the complex manifold $\Stab^{\dag}(S)$ one can find a $\bar{v}_1$-generic stability condition $\sigma$. Then the moduli space $\tilde{\modulis}_{t,a}:=\modulis_{\sigma}(\bar{v}_1)$ is a $\KK^{[2]}$-type hyperkähler manifold and one has a flop contraction 
$$\tilde{\modulis}_{t,a}\rightarrow\modulis_{\sigma_{t,a}}(\bar{v}_1)$$
induced by wall-crossing whose contracted locus is a union of some $\bp^2$.
\end{proof}
\end{proposition}

It is also expected that $\tilde{\modulis}_{t,a}$ is birational to $\modulis_H(\bar{v}_1)$ by showing that there are no totally semistable walls along $\sigma_{\lambda/t H,-a/t H}$ for $\lambda\geq 1$.

\begin{example}\label{11_main-result-1-lemma2-example}
Suppose in addition that $S$ has Picard rank one, then 
$$\ker Z_{t,a}=\bz(t,-aH,t(n-1))$$
and $\modulis_{\sigma_{t,a}}(\bar{v}_1)$ is not a hyperkähler manifold only when $t=5$ and $n=2$. In this case, one falls into the situation of \cite[Section 6.5]{Liu26a} and \cite[Section 5]{LZ26}.
\end{example}

Now for any $(S,H)$ in $\MBE(2t)_a$, one has a stability condition $\sigma_{t,a}$ and the unique Hodge isometry $\phi$ in $\HR(2t)$ fixing its central charge $Z_{t,a}$. So one can find a non-symplectic involution $\Pi$ on $\cate{D}^b(S)$ realizing $\phi$ using Theorem \ref{02_K3-Hu-conway}.

\begin{theorem}[Theorem \ref{00_main-result-1}]\label{11_main-result-1}
Choose $t\geq 2$, the locus $\MBE(2t)$ in $\MKK(2t)$ is the union of open loci $\MBE(2t)_a$ described in Proposition \ref{11_main-result-1-lemma1}.

\begin{proof}
Any non-empty open subset $B\subset\MKK(2t)$ contains Picard rank one K3 surfaces, so one can see the inclusion
$$\MBE(2t)\subseteq\bigcup_{a^2+1\equiv0\pmod{t} \atop a\in[1,t]} \MBE(2t)_a$$
from Definition \ref{00_main-definition} using Proposition \ref{11_involution-description}. 

Conversely, one considers the non-symplectic involution $\Pi$ on $\DbS$ defined uniformly for polarized K3 surfaces in $\MBE(2t)_a$ as above. Then we claim that the involution $\Pi$ defines an $\Stwo$-action on $\DbS$, and it suffices to show that $\Pi$ fixes an object in $\cate{D}^b(S)$ satisfying $\Hom(E,E)=\bc$ by virtue of Proposition \ref{02_Enriques-group-action-cri}.

According to Proposition \ref{11_main-result-1-lemma2} and \cite{Bea11}, the involution $\Pi$ does not fix a simple object in the moduli space $\modulis_{\sigma_{t,a}}(\bar{v}_1)$ only if $\modulis_{\sigma_{t,a}}(\bar{v}_1)$ is singular and $\Pi$ only fixes some strictly semistable points in it. Such a point $p$ is represented by $A\oplus B$ for two spherical objects such that $\Pi(A)\cong B$ and $\Ext^1(A,B)\cong \Ext^1(B,A)\cong\bc^3$.

In the remainder of this argument, we will show that the involution $\Pi$ cannot fix only some strictly semistable points in $\modulis_{\sigma_{t,a}}(\bar{v}_1)$.

By virtue of \cite[Theorem 1.1]{AS18} and \cite[Corollary 4.1]{AS25}, the local analytic model of $\modulis_{\sigma_{t,a}}(\bar{v}_1)$ at $p$ is the standard 4-dimensional symplectic singularity
$$\{(x,y)\in\Ext^1(B,A)\oplus\Ext^1(B,A)^\vee\,|\,y(x)=0\}\agit\bc^\times$$
where the $\bc^\times$-action is given by $\lambda\cdot(x,y)=(\lambda x,\lambda^{-1}y)$.

By construction, the involution $\Pi$ induces a linear isomorphism
$$g\colon \Ext^1(B,A)\rightarrow\Ext^1(B,A)^\vee\cong\Ext^1(A,B)$$
and it acts on the local model around $p$ by $(x,y)\mapsto(g^{-1}y,gx)$. So one can see that the fixed points are described by the quadratic cone
$$\{x\in \Ext^1(B,A)\,|\,g(x)(x)=0\}/\{\pm1\}$$
which is non-trivial as $\Ext^1(B,A)\cong\bc^3$. It means that $\Pi$ has to fix a stable point in the moduli space $\modulis_{\sigma_{t,a}}(\bar{v}_1)$ near $p$, a contradiction.
\end{proof}
\end{theorem}

\begin{remark}
It is expected that the involution on $\modulis_{\sigma_{t,a}}(\bar{v}_1)$ induced by $\Pi$ fixes an irreducible singular surface, similar to \cite[3.4]{Bea11}. 
\end{remark}

According to Theorem \ref{02_FL-result}, the locus $\MBE(2t)$ is non-empty only when $t$ is neither divisible by $4$ nor by an odd prime having the form $4k+3$.

\begin{proposition}\label{11_BE-locus-degree-4}
The open locus $\MBE(4)\subset \MKK(4)$ coincides with the locus of quartic K3 surfaces under the canonical polarization induced from $\bp^3$. 

\begin{proof}
One computes using Theorem \ref{11_main-result-1} that $\MBE(4)=\MKK(4)-\NL_{2,0}^4$. So the desired assertion is equivalent to the statement: an ample line bundle $H$ on a K3 surface $S$ with $H^2=4$ is very ample if and only if $(S,H)$ is not in $\NL_{2,0}^4$. This is nothing but Reider's method for degree $4$ K3 surfaces \cite[pages 35-37]{MoL}.
\end{proof}
\end{proposition}

According to \cite{KP17}, one has the identification
$$\DbS_{\Stwo}\cong\schb_Y$$
where $\schb_Y:=\langle\sho_Y,\sho_Y(1)\rangle^{\perp}\subset\cate{D}^b(Y)$ is the \emph{Kuznetsov component} of the double cover $f\colon Y\rightarrow\bp^3$ branched over the quartic surface $S$ with $\sho_Y(1)=f^*\sho_{\bp^3}(1)$. The threefold $Y$ is usually called the \emph{quartic double solid} over $S$.

\subsection{The families of group actions}
In this subsection, we will choose integers $t\geq 2$ and $1\leq a\leq t$ such that $\MBE(2t)_a$ is non-empty, and consider the stability condition $\sigma_{t,a}:=\sigma_{H/t,-aH/t}$ on any $(S,H)$ in $\MBE(2t)_a$.

\begin{proposition}\label{11_M(t,-aH,a^2)-property}
Consider some polarized K3 surface $(S,H)$ in $\MBE(2t)_a$, then the moduli space $\modulis_H(t,-aH,a^2)$ of semistable sheaves on $S$ is a fine moduli space and contains only stable vector bundles.

\begin{proof}
To begin with, one can check that $\modulis_H(t,-aH,a^2)$ is a fine moduli space using the fact that $a$ and $t$ are coprime and the criterion \cite[Theorem 4.6.5]{HL}.
	
Then one notices that the moduli space $\modulis_H(t,-aH,a^2)$ does not contain strictly semistable sheaves. Otherwise, one chooses a strictly semistable $\shh$ and a stable factor $\shg$ thereof with $v(\shg)=(r,\Delta,s)$ for some $0<r<t$. Then
$$\frac{\chi(\shg)}{r}=\frac{\chi(\shh)}{t}$$
which means that $st=a^2r$. So one has $t|a^2r$. However, the integers $a$ and $t$ are coprime since $a^2+1\equiv0\pmod{t}$. It means that $t|r$, absurd. 

In the end, we will check that any $\shh$ in $\modulis_H(t,-aH,a^2)$ is a vector bundle. The sheaf $\shh$ is in particular slope semistable. So it is $\sigma_{t,a}$-semistable with phase $0$ by construction. One notices that any nonzero $\sigma_{t,a}$-semistable sheaf with phase $0$ has integral central charge. Thus $\shh$ is $\sigma_{t,a}$-stable as $Z_{t,a}(\shh)=1$ and is therefore a vector bundle according to \cite[Lemma 10.1(b)]{Bri08}.
\end{proof}
\end{proposition}

Also, one makes the following important observation.

\begin{corollary}\label{11_M(t,-aH,a^2)-stab(S)}
Consider a polarized K3 surface $(S,H)$ in $\MBE(2t)_a$, then
$$\modulis_{\sigma_{\lambda H/t,-aH/t}}(-t,aH,-a^2)=\modulis_H(t,-aH,a^2)[1]$$
for any $\lambda\geq 1$.

\begin{proof}
One notices that $Z_{\lambda H/t,-aH/t}(-t,aH,-a^2)=-\lambda^2$, then the desired statement follows from Proposition \ref{11_M(t,-aH,a^2)-property} and \cite[Lemma 10.1~(b)]{Bri08}.
\end{proof}
\end{corollary}

Moreover, one can see that the moduli space is isomorphic to $S$.

\begin{proposition}\label{11_M(t,-aH,a^2)-iso-S}
One has $S\cong\modulis_H(t,-aH,a^2)$ for any $(S,H)$ in $\MBE(2t)_a$.

\begin{proof}
According to Theorem \ref{11_bridgeland-hyperk-K3} and Proposition \ref{11_M(t,-aH,a^2)-property}, the moduli space 
$$M:=\modulis_H(t,-aH,a^2)$$ 
is a K3 surface such that $\HH^2(M,\bz)\cong v_0^{\perp}/\bz\cdot v_0$ for the vector $v_0=(t,-aH,a^2)$. 

Suppose that $S$ has Picard rank one, then 
$$\modulis_H(t,-aH,a^2)\cong \modulis_H(t,H,1)\cong \modulis_H(1,H,t)\cong S$$
where the first isomorphism is due to \cite[Proposition 2.1]{MMY20} and the second one can be checked directly. In particular, one has an isometry
$$ \varphi\colon\HH^2(S,\bz)\rightarrow \HH^2(M,\bz)\cong v_0^{\perp}/\bz\cdot v_0=\langle u_0,v_0\rangle^{\perp}\subset \MukaiZ{S}$$
where $u_0=(1,0,-m)$ for the integer $m:=(a^2+1)/t$. This isometry $\varphi$ can be described uniformly as $\kappa\mapsto(0,-\kappa+m(\kappa\cdot H)H,0)-a(\kappa\cdot H)(1,0,m)$.

Since the locus of Picard rank one K3 surfaces is dense in the connected open locus $\MBE(2t)_a$ and $\modulis_H(t,-aH,a^2)$ is fine for any $(S,H)$ in $\MBE(2t)_a$, one can spread the isometry $\varphi$ to any element in $\MBE(2t)_a$ using the relative universal family with respect to the corresponding relative moduli space over $\MBE(2t)_a$.
\end{proof}
\end{proposition}

It allows us to reach the following generalization of \cite[Proposition 3.16]{Liu26b}.

\begin{proposition}\label{11_involution-absolute}
The involution $\Pi$ fixing the stability condition $\sigma_{t,a}$ as in Proposition \ref{11_involution-description} is isomorphic to the Fourier--Mukai transform $\Phi_{\shu[1]}$ where $\shu$ is the universal family on $S\times \modulis_H(t,-aH,a^2)$.

\begin{proof}
Using the argument of \cite[Corollary 10.12]{Hu}, there exists a universal family $\shu$ on $S\times M$ for the fine moduli space $M:=\modulis_H(t,-aH,a^2)$ such that the Fourier--Mukai transform $\Psi:=\Phi_{\shu[1]}$ induces the same Hodge isometry as $\Pi^{\coho}$. Here we use the fact that $S\cong M$ as in Proposition \ref{11_M(t,-aH,a^2)-iso-S}. The autoequivalence $\Psi$ preserves $\Stab^{\dag}(S)$ due to \cite[Proposition 4.2]{Hu08}. Moreover, one checks
$$\Psi.\sigma_{t,a}=\sigma_{t,a}$$
by inspecting the central charges and applying Corollary \ref{11_M(t,-aH,a^2)-stab(S)}. Therefore one can conclude $\Psi\cong\Pi$ using Theorem \ref{02_K3-Hu-conway}.
\end{proof}
\end{proposition}

This property can be globalized into a family as follows.

\begin{theorem}[Theorem \ref{00_main-result-1.2}]\label{11_main-result-1.2}
Consider an integer $t\geq 2$ and $1\leq a\leq t$ such that the open locus $\MBE(2t)_a$ is non-empty, then for any family $p\colon\schs\rightarrow B$ of polarized K3 surfaces over a smooth connected variety with fibers in $\MBE(2t)_a$ one can define an autoequivalence $\Pi$ on the category of perfect complexes $\cate{D}_p(\schs)$, such that $\Pi^2\cong p^*\shl$ for some line bundle $\shl$ on $B$ and the base change of $\Pi$ for each point $b\rightarrow B$ induces the $\Stwo$-action on $\cate{D}^b(\schs_b)$ in Definition \ref{00_main-definition} and fixes $\sigma_{t,a}$.

\begin{proof}
The morphism $\schs\rightarrow B$ is projective, so one can consider the moduli space of relative semistable sheaves on $\schs$ as in \cite[Theorem 4.3.7]{HL}.

One takes the relative ample polarization $\schh$ for $\schs\rightarrow B$ and considers the moduli space $\schm=\modulis_{\schh}(t,-a\schh,a^2)$. Then there exists a relative universal family 
$$\schu\in\cate{D}^b(\schs\times_B\schm)$$
according to the construction of $\schm$ and Proposition \ref{11_M(t,-aH,a^2)-property}.

The family $\mathcal U$ induces a relative Mukai homomorphism whose fiber is the Mukai isometry used in Proposition
\ref{11_M(t,-aH,a^2)-iso-S}. Equivalently, up to taking the inverse according to
the convention for pullbacks, it is described by the uniform formula
$$
\kappa\longmapsto
(0,-\kappa+m(\kappa\cdot H)H,0)-a(\kappa\cdot H)(1,0,m)
$$
for $m:=(a^2+1)/t$. Since $\mathcal H$ is a relative class and the isometry is constant, these fiberwise isometries are monodromy-compatible. Thus they form a single global isometry of integral variations of Hodge structures $
\lambda\colon\deriver^2q_*\bz
\stackrel{\sim}{\longrightarrow}
\deriver^2p_*\bz$ for the two families $p\colon\schs\rightarrow B$ and $q\colon\schm\rightarrow B$ of polarized K3 surfaces.

This isometry $\lambda$ induces an isomorphism $f\colon \schs\rightarrow\schm$ such that $f^*_b=\lambda_b$ using the standard polarized relative isomorphism functor and the global Torelli theorem.

After normalization, the universal family $\schu$ induces a Fourier--Mukai functor
$$
\Pi:=\Phi_{\mathcal U[1]}\colon
\cate{D}_p(\mathcal S)\longrightarrow\cate{D}_p(\mathcal S)
$$
such that the base change $\Pi_b$ thereof on $\cate{D}^b(\schs_b)$ for any $b\rightarrow B$ is the involution in Proposition \ref{11_involution-absolute} which induces the desired $\Stwo$-action on $\cate{D}^b(\schs_b)$ and fixes $\sigma_{t,a}$. 

One checks that $\Pi$ is an autoequivalence as in \cite[Proposition 2.15]{HLS09} and the composite $\Pi^2$ is isomorphic to the tensor product with $p^*\shl$ for some line bundle $\shl$ on $B$ by adapting the argument for \cite[Corollary 5.23]{Hu}.
\end{proof}
\end{theorem}

\begin{remark}\label{11_main-result-1.2-rem}
Consequently, the involution $\Pi$ étale locally induces an $\Stwo$-action on the family $\cate{D}_p(\schs)$ according to \cite[Proposition 3.9]{BP23}.
\end{remark}

The author has been informed that this local action can be made global for families of polarized K3 surfaces in $\MBE(2t)_a$ by the upcoming work \cite{BPPZ26}. 

\subsection{The universal BE general K3 surface}
To finish the section, we discuss a possible universal family of K3 surfaces in $\MBE(2t)_a$ for $t\geq 2$. Here one recalls that $\MBE(2t)_a$ is only a coarse moduli space (see e.g.~\cite{HuK3}) so it does not admit a universal family in the usual sense. In this case, we need to instead use universal families for the Hilbert schemes of the projective spaces.

\begin{proposition}
The polarization $H$ is very ample for any Bridgeland--Enriques general K3 surface $(S,H)$ in $\MBE(2t)$ if $t\geq 2$.

\begin{proof}
According to \cite{MoL,Rei88,Sa74} and the fact that $t$ is not divisible by $4$, the ample class $H$ is not very ample only if there exists a divisor class $E$ in $\NS(S)$ satisfying $E^2=0$ and $E\cdot H$ is $1$ or $2$. It is impossible by Proposition \ref{11_main-result-1-lemma1}.
\end{proof}
\end{proposition} 

Therefore, the polarization $H$ induces an embedding
$$S\hookrightarrow|H|^\vee=\bp^{t+1}$$
for any $(S,H)$ in $\MBE(2t)$ with $t\geq 2$. The corresponding Hilbert polynomial is
$$P_{2t}(m)=\chi(S,\sho_S(mH))=tm^2+2$$
and one can define an open subscheme $$B_{2t,a}^{\BE}\subset\Hilb_{tm^2+2}(\bp^{t+1})$$ 
parameterizing the Bridgeland--Enriques general K3 surfaces in $\MBE(2t)_a$ together with the embedding information. This gives a universal object
$$\schs_{2t,a}^{\BE}\subset\bp^{t+1}\times B_{2t,a}^{\BE}$$
which is the restriction of the tautological Hilbert subscheme. 

\begin{example}
The universal object for $\MBE(4)$ is the universal quartic surface.
\end{example}

The universal object for $\MBE(2t)_a$ can be endowed with an autoequivalence as in Theorem \ref{11_main-result-1.2}. However, it is difficult to write down a clear description of the universal object when the degree $2t$ is larger than $4$. In fact, there has already been a degeneration phenomenon on $\MBE(10)$.

\section{Degeneration of special Gushel--Mukai Threefolds}\label{000_Section-12}
In this section, we will discuss the categorical degeneration phenomenon encoded in the open locus $\MBE(10)$.

\subsection{The categorical degeneration}
To describe the categorical degeneration one can get from Theorem \ref{00_main-result-2.2}, we need to figure out $\MBE(10)$.

\begin{proposition}[Proposition \ref{00_main-result-2.1}]\label{12_main-result-2.1}
A polarized K3 surface $(S,H)$ in $\MBE(10)$ belongs to one of the following two cases:
\begin{enumerate}
	\item\label{12_main-result-2.1(1)} a Brill--Noether general surface of degree $10$ which is not in $\NL_{4,0}^{10}$;
	\item\label{12_main-result-2.1(2)} a quartic K3 surface containing an elliptic curve $E$ such that $H\cdot E=3$.
\end{enumerate}
Conversely, any polarized K3 surface in (1) is contained in $\MBE(10)$, and a polarized K3 surface in (2) is in $\MBE(10)$ if one has $\sho_S(H-E)\cong\sho_{\bp^3}(1)|_{S}$.

\begin{proof}
One computes using Theorem \ref{11_main-result-1} that
$$\MBE(10)=\MKK(10)-\NL_{4,0}^{10}$$
then by \cite[Lemma 2.8]{GLT15} a K3 surface not in (1) must be in $\NL_{3,0}^{10}-\NL_{4,0}^{10}$. One chooses a K3 surface $(S,H)$ in $\NL_{3,0}^{10}-\NL_{4,0}^{10}$ and a class $E\in\NS(S)$ with $E^2=0$ and $E\cdot H=3$, it reduces to show that $D:=H-E$ is very ample.

One first sees that $D$ is ample from the observation that $L:=D-E$ is the class of a rational curve. Then, according to Proposition \ref{11_BE-locus-degree-4}, it suffices to check that there is no class $\Delta\in\NS(S)$ satisfying $\Delta^2=0$ and $\Delta\cdot D=2$. Otherwise, one finds such a class $\Delta$ and divides into cases according to whether $\Delta$ or $-\Delta$ is effective. 

Suppose that $-\Delta$ is effective, then 
\begin{center}
	$\Delta\cdot(H-2E)<0$ and $\Delta\cdot E<0$
\end{center}
as $L:=H-2E$ is effective as well. Therefore, one cannot have $\Delta\cdot D=2$. 

Suppose that $\Delta$ is effective, then $\Delta\cdot E=\Delta\cdot L=1$ due to $\Delta\cdot D=2$. So one can see $\Delta\cdot E=1$ and $\Delta\cdot H=3$. However, it will lead to $(D-\Delta)^2=0$ and $(D-\Delta)\cdot H=4$ since $D\cdot H=7$, which contradicts $(S,H)\notin\NL_{4,0}^{10}$.
\end{proof}
\end{proposition}

The polarized K3 surface $(S,H)$ in Proposition \ref{12_main-result-2.1}~\eqref{12_main-result-2.1(1)} is called a \emph{strongly smooth Gushel--Mukai surface} as in \cite{Beri25,DK18} and vice versa. The very ample class $H$ embeds it into a smooth Fano threefold $M=\Gr(2,5)\cap\bp^6$ such that $S\in|\sho_M(2)|$. 
 
The \emph{special Gushel--Mukai threefold} $X$ over a strongly smooth Gushel--Mukai surface $S$ is the double cover $X\rightarrow M$ branched over the surface $S$. The inverse image $\sho_X(1)$ of the ample line bundle $\sho_{\bp^6}(1)|_M$ along the double cover is a degree ten ample line bundle, and the inverse image $\shu_X$ of the restricted tautological bundle $\shu_{\Gr(2,5)}|_M$ along the double cover is a slope stable vector bundle.

\begin{definition}\label{12_Kuz-special-GM-3-defi}
Let $X$ be a special Gushel--Mukai threefold over a strongly smooth Gushel--Mukai surface $S$, then the \defi{Kuznetsov component} of $X$ is the admissible subcategory $\scha_X:=\langle\sho_X,\shu^\vee_X\rangle^{\perp}\subset\DbX$.
\end{definition}

According to \cite{KP17}, one has $\DbS_{\Stwo}\cong\scha_X$ such that the residual $\Stwo^\vee$-action on $\scha_X$ is induced by the covering involution of the double cover $X\rightarrow M$.

Now let us recall the following theorem stated in the introduction, which will be proved in the next subsection.

\begin{theorem}[Theorem \ref{00_main-result-2.2}]
Consider a K3 surface $S$ in Proposition \ref{12_main-result-2.1}~\eqref{12_main-result-2.1(2)}, then the Enriques category $\DbS_{\Stwo}$ coming from $\MBE(10)$ is equivalent to the Kuznetsov component $\schb_Y$ of the quartic double solid $Y$ over $S$.
\end{theorem}

One takes a K3 surface $(S,H)$ as in Proposition \ref{12_main-result-2.1}~\eqref{12_main-result-2.1(2)}, then the class $L:=H-2E$ is a line with respect to the quartic surface section $D:=H-E$. One blows up the quartic double solid $Y$ containing $S$ along the line $L$ and denotes by $\tilde{L}$ the exceptional divisor. Then the linear system $\mathfrak{d}:=|\tilde{D}-\tilde{L}|$ on $\tilde{Y}$ is a pencil whose base locus $\tilde{C}$ is a smooth rational curve such that $\tilde{D}\cdot\tilde{C}=1$ and $\tilde{L}\cdot\tilde{C}=2$.

According to \cite{Mel99,PCS05}, the linear system $|2\tilde{D}-\tilde{L}|$ defines a contraction
$$\tilde{Y}\rightarrow X_{\star}\subset\bp^7$$
onto a degenerate Gushel--Mukai threefold $X_{\star}$ (in the sense that it is also a Fano threefold of genus $6$) with exceptional locus $\tilde{C}$.

\begin{theorem}[Kuznetsov--Shinder \cite{KS25}]
The Kuznetsov component $\schb_Y$ of the quartic double solid $Y$ is equivalent to a subcategory $\bar{\scha}_{X_{\star}}\subset\cate{D}^b(X_{\star})$.
\end{theorem}

The subcategory $\bar{\scha}_{X_{\star}}$ is called the \emph{Kuznetsov component} of the degenerate Gushel--Mukai threefold $X_{\star}$. Using Theorem \ref{11_main-result-1.2} and Remark \ref{11_main-result-1.2-rem}, one can find for a suitable small base $B\subset\MBE(10)$ a family of Enriques categories
$$\cate{D}_p(\schs_B)_{\Stwo}\rightarrow B$$
whose fiber $\cate{D}^b(\schs_b)_{\Stwo}$ at $b\in B$ is either equivalent to the Kuznetsov component of a special Gushel--Mukai threefold or that of a degenerate one.

\begin{remark}
The degenerate threefold $X_{\star}$ above has another small resolution
$$f\colon W\rightarrow X_{\star}$$
such that $W\subset\bp(\sho_{\bp^1}\oplus\sho_{\bp^1}^{\oplus 2}(1)\oplus\sho_{\bp^1}(2))$ is a cubic del Pezzo fibration and the exceptional locus of $f$ is the minimal section of $\bp(\sho_{\bp^1}\oplus\sho_{\bp^1}^{\oplus 2}(1)\oplus\sho_{\bp^1}(2))$ contracted by the anti-canonical class $-K_W$. A general hyperplane section of $W$ is a K3 surface
$$T\subset\bp(\sho_{\bp^1}^{\oplus 2}(1)\oplus\sho_{\bp^1}(2))$$
admitting a planar cubic elliptic fibration. Under the canonical polarization, one can see that $T$ is in $\NL_{3,0}^{10}$. It would be interesting to check whether $T$ belongs to $\MBE(10)$ and investigate the relation between $\DbS$ and $\cate{D}^b(T)$. The author hopes that it will be helpful to find a direct proof of $\cate{D}^b(S)_{\Stwo}\cong\bar{\scha}_{X_{\star}}$.
\end{remark}

\subsection{Conjugation of involutions}
Here, we will fix a polarized K3 surface $(S,H)$ in $\MBE(10)$ as in Proposition \ref{12_main-result-2.1}~\eqref{12_main-result-2.1(2)} and let $E$ be the elliptic class.

Then one has two special stability conditions
\begin{center}
$\sigma_{5,2}:=\sigma_{\frac{1}{5}H,-\frac{2}{5}H}$ and $\sigma_{2,1}:=\sigma_{\frac{1}{2}D,-\frac{1}{2}D}$
\end{center}  
on $\DbS$ coming from $(S,H)\in\MBE(10)$ and $(S,D)\in\MBE(4)$ respectively.

The involution $\Pi_{5,2}$ generating the $\Stwo$-action on $\DbS$ for $(S,H)\in\MBE(10)$ induces the Hodge isometry $\phi_{5,2}$ on $\MukaiZ{S}$ such that
$$\phi_{5,2}(r,\Delta,s)=(-4r-5s-2\Delta\cdot H,(2r+2s+\Delta\cdot H)H-\Delta,-5r-4s-2\Delta\cdot H)$$
according to \cite[Proposition 6.1]{Liu26a}.

The involution $\Pi_{2,1}$ generating the $\Stwo$-action on $\DbS$ for $(S,D)\in\MBE(4)$ induces the Hodge isometry $\phi_{2,1}$ on $\MukaiZ{S}$ such that
$$\phi_{2,1}(r,\Delta,s)=(-r-2s-\Delta\cdot D,(r+s+\Delta\cdot D)D-\Delta,-2r-s-\Delta\cdot D)$$
according to \cite[Proposition 5.1]{Liu26a}.

\begin{proposition}\label{12_main-result-2.2-lemma}
The object $\shi_x(-E)$ is stable with respect to $\sigma_{5,2}$.
	
\begin{proof}
Otherwise, one can choose a short exact sequence
$$0\rightarrow A\rightarrow\shi_x(-E)\rightarrow B\rightarrow0$$
in $\scha_{5,2}$ with $\phi(A)\geq\phi(\shi_x(-E))$. The long cohomological exact sequence 
$$0\rightarrow\shh^{-1}(B)\rightarrow\shh^0(A)\rightarrow\shi_x(-E)\rightarrow\shh^0(B)\rightarrow0$$
implies that $A$ is a torsion-free sheaf. In particular, one has $\Im Z_{5,2}(A)>0$. It follows that $\Im Z_{5,2}(B)=0$ and hence $\phi(A)<\phi(\shi_x(-E))$, a contradiction.
\end{proof}
\end{proposition}

Now we are able to conclude the main result of this section. 

\begin{theorem}[Theorem \ref{00_main-result-2.2}]\label{12_main-result-2.2}
Consider a K3 surface $S$ in Proposition \ref{12_main-result-2.1}~\eqref{12_main-result-2.1(2)}, then the Enriques category $\DbS_{\Stwo}$ coming from $\MBE(10)$ is equivalent to the Kuznetsov component $\schb_Y$ of the quartic double solid $Y$ over $S$.

\begin{proof}
It suffices to show that the involutions $\Pi_{5,2}$ and $\Pi_{2,1}$ on $\DbS$ are conjugate to each other. According to Proposition \ref{12_main-result-2.2-lemma}, the stability condition $$\sigma':=\funct{T}_{\sho_S(-E)}^{-1}.\sigma_{5,2}=(\schq,W)$$ 
is geometric, so that $\sigma'$ is in $\Stab^{\dag}(S)$. Then one can find a unique $\tilde{g}\in\grp$ such that $\sigma'.\tilde{g}=\sigma_{\omega,\beta}$ for some ample class $\omega$. One computes that 
$$W(r,\Delta,s)=-\left(\frac{1}{5}r+\frac{2}{5}\Delta\cdot D+\frac{3}{5}s\right)+\left(\frac{3}{5}r+\frac{1}{5}\Delta\cdot D-\frac{1}{5}s\right)\sqrt{-1}$$
so a direct inspection shows that $\sigma_{\omega,\beta}=\sigma_{2,1}$. In particular, one has 
$$\Pi'.\sigma_{5,2}=\sigma_{5,2}$$
for the involution $\Pi':=\funct{T}_{\sho_S(-E)}\circ\Pi_{2,1}\circ\funct{T}^{-1}_{\sho_S(-E)}$.

Moreover, one checks that the Hodge isometry induced by $\Pi'$ is the same as $\phi_{5,2}$. Hence, one concludes $\Pi'\cong\Pi_{5,2}$ due to Theorem \ref{02_K3-Hu-conway}.
\end{proof}

\end{theorem}

\section{Towards Hodge-Special Gushel--Mukai Fourfolds}\label{000_Section-13}

In this section, we will point out the relation between Hodge-special Gushel--Mukai fourfolds and Bridgeland--Enriques general K3 surfaces.

\subsection{Kuznetsov components of GM fourfolds}
Let $V_5$ be a five-dimensional vector space, let $\Gr(2,V_5)$ be the Grassmannian of two-dimensional subspaces of $V_5$, viewed in $\bp(\land^2V_5)\cong\bp^9$ via the Plücker embedding, and let $\CGr(2,V_5)$ in $\bp(\bc\oplus\land^2V_5)\cong\bp^{10}$ be the cone over $\Gr(2, V_5)$ of vertex $\nu:= \bp(\bc)$.

\begin{definition}
A \defi{Gushel--Mukai fourfold} is a smooth $4$-dimensional intersection $X=\CGr(2,V_5)\cap \bp(W)\cap Q$ where $W\subset \bc\oplus\land^2 V_5$ is a nine-dimensional subspace and $Q\subset\bp(W)$ is a quadric hypersurface.
\end{definition}

The vertex $\nu$ of the cone $\CGr(2,V_5)$ does not belong to $X$ as $X$ is smooth. So the linear projection from $\nu$ defines a regular map
$$\gamma_X\colon X\rightarrow\Gr(2,V_5)$$
called the \emph{Gushel map}. The geometry of $\gamma_X$ splits $X$ into two possible types.

Suppose that $\nu\in\bp(W)$, then $\bp(W)$ is a cone over $\bp(W/\bc)$ and $X$ is a double cover of the linear section $\bp(W/\bc)\cap\Gr(2,V_5)$ branched along its intersection with a quadric. In this case, we say that $X$ is \emph{special}.

Suppose that $\nu\notin\bp(W)$, then $X$ is isomorphic to a quadric section of the linear section of $\Gr(2,V_5)$ given by the projection of $\bp(W)$ to $\bp(\land^2 V_5)$. In this case, we say that $X$ is \emph{ordinary} and one sees that $\gamma_X$ is an embedding.

In any case, one has an ample line bundle $\sho_X(1):=\gamma_X^*\sho_{\Gr(2,V_5)}(\sigma_{1,0})$ for the Schubert cycle $\sigma_{1,0}$ and a vector bundle $\shu_X:=\gamma_X^*\shu_{\Gr(2,V_5)}$.

\begin{definition}[\cite{Ku06,KP18}]
Let $X$ be a Gushel--Mukai fourfold, then the subcategory
$$\schc_X:=\{A\in\DbX\,|\,\deriver\Hom(\sho_X(i),A)=\deriver\Hom(\shu^\vee_X(i),A)=0\textup{ for }i=0,1\}$$
is defined as its \defi{Kuznetsov component}.
\end{definition}

The Kuznetsov component $\schc_X$ of a Gushel--Mukai fourfold $X$ is a 2-Calabi--Yau category according to \cite[Proposition 2.6]{KP18}. Similar to the Mukai lattice for K3 surfaces, one can define a pure weight-two Hodge structure 
$$\MukaiZ{\schc_X}:=\{\kappa\in K_{\topo}(X)\,|\,\chi([\sho_X(i)],\kappa)=\chi([\shu_X^\vee(i)],\kappa)=0\textup{ for }i=0,1\}$$
such that $\Mukai\,\!^{2,0}(\schc_X)=\HH^{3,1}(X)\cap\Mukai(\schc_X)$ and $\Mukai\,\!^{1,1}(\schc_X)=\bigoplus_i\HH^{i,i}(X)\cap\Mukai(\schc_X)$. 

The polarization on this lattice is given by the pairing $\langle x,y\rangle=-\chi(x,y)$. Similar to K3 surfaces, one has $\MukaiZ{\schc_X}\cong \ExL_8^2(-1)\oplus \HyU^4$ by \cite[Section 3.1]{Per19}.

Moreover, one can always find two classes $\lambda_1,\lambda_2$ in the lattice $\MukaiZ{\schc_X}$ which span a sublattice isometric to $\Ala^2_1$. In fact, one has
$$\ch(\lambda_1)=-2+(h^2-\sigma)-\frac{1}{20}h^4\quad\textup{and}\quad\ch(\lambda_2)=-4+2h-\frac{1}{6}h^3$$
where $\sigma:=\gamma_X^*\sigma_{2,0}\in \HH^4(X,\bz)$ for the Schubert cycle $\sigma_{2,0}$.

\begin{theorem}[Pertusi \cite{Per19}]
One has $\kappa\in \langle\lambda_1,\lambda_2\rangle^{\perp}$ if and only if $\ch(\kappa)$ is in
$$\HH^4(X,\bq)_{00}=\{x\in \HH^4(X,\bq)\,|\,x\cdot h^2=x\cdot\sigma=0 \}.$$
Also, $v:=\ch(-)\sqrt{\td(X)}$ induces a Hodge isometry $\langle\lambda_1,\lambda_2\rangle^{\perp}\stackrel{\sim}{\rightarrow}\HH^4(X,\bz)_{00}(1)$.
\end{theorem}

Moreover, one has a canonical $\Stwo$-action on $\schc_X$ due to \cite{KP17}. It is direct to compute that the involution inducing this $\Stwo$-action induces a Hodge isometry $\psi$ on $\MukaiZ{\schc_X}$ such that $\psi(\lambda_i)=\lambda_i$ for $i=1,2$ and $\psi(x)=-x$ for $x\in\HH^4(X,\bq)_{00}$.

\subsection{The first-type Hodge-special GM fourfolds}
Although not necessary, it is sufficient to treat only the ordinary Gushel--Mukai fourfolds. 

\begin{proposition}[{{\cite[Lemma 3.8]{KP18}}} and {{\cite[Corollary 6.5]{KP23}}}]
The Kuznetsov component of a special Gushel--Mukai fourfold is equivalent to the Kuznetsov component of an ordinary Gushel--Mukai fourfold.
\end{proposition}

In general, a Gushel--Mukai fourfold $X$ is called \emph{Hodge-special} when $\HH^{2,2}(X,\bz)$ contains a rank-three primitive sublattice $\Lambda$ containing 
$$\gamma_X^*\HH^4(\Gr(2,V_5),\bz)$$ 
for $\gamma_X\colon X\rightarrow\Gr(2,V_5)$. The \emph{discriminant} of a Hodge-special $X$ is defined to be the discriminant $\disc(\Lambda)=2t$, where $t\geq 5$ and $t\equiv 0,1,2\pmod{4}$ due to \cite{DIM15}. 

\begin{definition}
A Hodge-special Gushel--Mukai fourfold $X$ is said to be of \defi{first-type} if its discriminant satisfies $t\equiv1,2\pmod{4}$. Otherwise, it is called a \defi{second-type} Hodge-special Gushel--Mukai fourfold.
\end{definition}

The following statement is a direct consequence of Theorem \ref{02_FL-result}~\eqref{02_FL-result(3)}, Theorem \ref{11_main-result-1}, and the definitions.

\begin{proposition}[Proposition \ref{00_main-result-3.1}]\label{13_main-result-3.1}
Consider a Gushel--Mukai fourfold $X$ such that there exists an equivalence $\schc_X\cong\DbS$ for some Bridgeland--Enriques general K3 surface $S$ that is compatible with the $\Stwo$-actions, then $X$ is a first-type Hodge-special Gushel--Mukai fourfold.
\end{proposition}

The full converse of this statement is not clear. Here let us provide a weak converse statement for general first-type Hodge-special Gushel--Mukai fourfolds.

\begin{definition}
A first-type Hodge-special Gushel--Mukai fourfold $X$ is said to be \defi{general} if the free abelian group $ \HH^{2,2}(X,\bz)$ has rank $3$.
\end{definition}

Using the results in \cite{PPZ22,Per19} about K3 surfaces associated with Gushel--Mukai fourfolds, it is not difficult to obtain the following assertion. 

\begin{theorem}[Theorem \ref{00_main-result-3.2}]\label{13_main-result-3.2}
Consider a general first-type Hodge-special Gushel--Mukai fourfold $X$ with discriminant $2t$, then for any non-empty open locus $\MBE(2t)_a$ one can find a K3 surface $S$ with Picard rank one lying in it such that there exists an equivalence $\schc_X\cong\DbS$ compatible with the $\Stwo$-actions.

\begin{proof}
One applies \cite[Theorem 3.6(1)]{Per19} to find a polarized K3 surface $(S,H)$ with Picard rank one and degree $2t$ and a Hodge isometry
$$\MukaiZ{\schc_X}\supset\Lambda^{\perp}\stackrel{\sim}{\rightarrow} H^{\perp}\subset \HH^2(S,\bz)$$
i.e.~the K3 surface $S$ is Hodge-associated with $X$. In particular, one has a primitive square zero class in $\Mukai\,\!^{1,1}(\schc_X,\bz)$. Then one has $\schc_X\cong\cate{D}^b(S)$ using \cite[Theorem 1.9]{PPZ22} and the categorical Torelli theorem for K3 surfaces. The equivalence is compatible with the $\Stwo$-actions on both sides by Theorem \ref{02_FL-result}. 

The $\Stwo$-action on $\cate{D}^b(S)$ comes from one of the non-empty $\MBE(2t)_a$. One can find an $\Stwo$-equivariant equivalence for the action coming from any other non-empty branch in $\MBE(2t)$ by passing to the Fourier--Mukai partners of $S$. One refers to \cite[Theorem 2.1]{HLOS03} for the construction of different Fourier--Mukai partners.
\end{proof}
\end{theorem}

In particular, one can find a stability condition $\sigma^X_{t,a}$ on $\schc_X$ through the equivalence such that $\modulis_{\sigma_{t,a}}(\bar{v}_i)\cong\modulis_{\sigma^X_{t,a}}(\lambda_i)$ or $\modulis_{\sigma_{t,a}}(\bar{v}_i)\cong\modulis_{\sigma^X_{t,a}}(\lambda_{3-i})$.

\begin{proposition}[Proposition \ref{00_main-result-4.1}]\label{13_main-result-4.1}
Consider a very general K3 surface $S$ in some open locus $\MBE(2t)_a$ for $t\geq10$, then the moduli spaces $\modulis_{\sigma_{t,a}}(\bar{v}_1)$ and $\modulis_{\sigma_{t,a}}(\bar{v}_2)$ are $\KK^{[2]}$-type hyperkähler manifolds with degree $2$.
	
\begin{proof}
The assertion about $\modulis_{\sigma_{t,a}}(\bar{v}_1)$ is explained in Example \ref{11_main-result-1-lemma2-example}. Now let us check it for $\modulis_{\sigma_{t,a}}(\bar{v}_2)$ when $\Pic(S)\cong\bz[H]$ following Proposition \ref{11_main-result-1-lemma2}.

Suppose that an element $E$ in $\modulis_{\sigma_{t,a}}(\bar{v}_2)$ is strictly semistable, then one can take a non-trivial Jordan--Hölder factor thereof with Mukai vector $w$.

According to Corollary \ref{11_BE-central-charge}, one must have
$$Z_{t,a}(\bar{v}_2-w)=Z_{t,a}(w)=\frac{1}{2}Z_{t,a}(\bar{v}_2)=-\frac{u}{t}+\frac{v}{t}\sqrt{-1}$$
since $u,v$ are coprime and $\Re Z_{t,a}(w),\Im Z_{t,a}(w)\in\frac{1}{t}\bz$ by construction. It follows that the strictly semistable object $E$ has exactly two distinct Jordan--Hölder factors say $A$ and $B$ with Mukai vector $w$ and $\bar{v}_2-w$ respectively, such that
$$w=\frac{\bar{v}_2+w_0}{2}$$
for some Mukai vector $w_0\in\ker Z_{t,a}\cap\Num(S)=\bz(t,-aH,t(n-1))$ where $n\geq 2$ is the integer satisfying $a^2-t(n-1)=-1$. One can see $w_0^2=-10$ from the argument of Proposition \ref{11_main-result-1-lemma2}. It is possible only when $t=5$.
\end{proof}
\end{proposition}

In this case, we will have two hyperkähler manifolds associated with $\schc_X$ for each equivariant equivalence $\schc_X\cong\cate{D}^b(S)$ with $S$ very general in $\MBE(2t)$.

\subsection{Moduli spaces and EPW sextics}
A Gushel--Mukai fourfold $X$ is canonically associated with two irreducible symplectic varieties due to \cite{DK18}, the double EPW sextic $\tilde{Y}_X$ and double dual EPW sextic $\tilde{Y}_X^\vee$. According to \cite{O'Grady13}, the variety $\tilde{Y}_X$ (or $\tilde{Y}_X^\vee$) always carries a non-symplectic involution and is a $\KK^{[2]}$-type hyperkähler manifold with a degree $2$ polarization when it is smooth.

\begin{theorem}[Theorem \ref{00_main-result-4.2}]\label{13_main-result-4.2}
Consider an equivariant equivalence $\schc_X\cong\DbS$ for some Gushel--Mukai fourfold $X$ and some Picard rank one K3 surface $S$ in $\MBE(2t)_a$ with $t\geq 10$ such that the double EPW varieties $\tilde{Y}_X$ and $\tilde{Y}^\vee_X$ are smooth, then $\modulis_{\sigma_{t,a}}(\bar{v}_i)$ is isomorphic to $\tilde{Y}_X$ or $\tilde{Y}^\vee_X$ for $i=1,2$. 

\begin{proof}
One has a bijective correspondence
$$\{\modulis_{\sigma_{t,a}}(\bar{v}_1),\modulis_{\sigma_{t,a}}(\bar{v}_2)\}\leftrightarrow\{\modulis_{\sigma^X_{t,a}}(\lambda_1),\modulis_{\sigma^X_{t,a}}(\lambda_2)\}$$ 
through the equivalence $\schc_X\cong\cate{D}^b(S)$. According to the argument for \cite[Proposition 5.17]{PPZ22}, one can show that $\modulis_{\sigma^X_{t,a}}(\lambda_1)$ is birational to $\tilde{Y}_X$ or $\tilde{Y}^\vee_X$. Here we will assume that $\modulis_{\sigma^X_{t,a}}(\lambda_1)\dashrightarrow \tilde{Y}_X$, then one has a parallel-transport Hodge isometry
$$\HH^2(\tilde{Y}_X,\bz)\rightarrow\HH^2(\modulis_{\sigma^X_{t,a}}(\lambda_1),\bz)\rightarrow\HH^2(\modulis_{\sigma_{t,a}}(\bar{v}_i),\bz)$$
for $i=1$ or $i=2$. By construction, the isometry will transport the square two ample class $h$ of $\tilde{Y}_X$ to $\theta_{\bar{v}_i}(\bar{v}_{3-i})$ up to a sign, where $\theta_{\bar{v}_i}$ is the Mukai isomorphism
$$\theta_{\bar{v}_i}\colon \bar{v}_i^{\perp}\stackrel{\sim}{\rightarrow}\HH^2(\modulis_{\sigma_{t,a}}(\bar{v}_i),\bz)$$
in \cite[Theorem 3.6]{BM14-MMP}. Then by virtue of the global Torelli theorem for hyperkähler manifolds, it suffices to check that $\theta_{\bar{v}_i}(\bar{v}_{3-i})$ is ample up to a sign as the Hodge isometry cannot send $h$ to an anti-ample class. It follows from the fact that the invariant sublattice in $\HH^2(\modulis_{\sigma_{t,a}}(\bar{v}_i),\bz)$ of the involution on $\modulis_{\sigma_{t,a}}(\bar{v}_i)$ induced by the $\Stwo$-action is $\bz[\theta_{\bar{v}_i}(\bar{v}_{3-i})]$ by \cite[Proposition 3.5]{Ou18} and Proposition \ref{11_BE-locus-square-2-vector}.

It means that one has $\modulis_{\sigma_{t,a}^X}(\lambda_1)\cong \tilde{Y}_X$ or $\modulis_{\sigma_{t,a}^X}(\lambda_1)\cong \tilde{Y}^\vee_X$. Similarly, one can see that $\modulis_{\sigma_{t,a}^X}(\lambda_2)\cong \tilde{Y}_X$ or $\modulis_{\sigma_{t,a}^X}(\lambda_2)\cong \tilde{Y}^\vee_X$. Consequently, one has
$$\{\modulis_{\sigma_{t,a}}(\bar{v}_1),\modulis_{\sigma_{t,a}}(\bar{v}_2)\}=\{\modulis_{\sigma^X_{t,a}}(\lambda_1),\modulis_{\sigma^X_{t,a}}(\lambda_2)\}\subseteq\{\tilde{Y}_X,\tilde{Y}^\vee_X\}$$
and therefore concludes.
\end{proof}
\end{theorem}

\begin{remark}
The Hilbert square of some Picard rank one K3 surface can be realized as a double EPW sextic by \cite[Corollary 7.6]{DM19}. 
\end{remark}

\begin{remark}
One can directly check that $-\theta_{\bar{v}_1}(\bar{v}_2)$ is an ample class on $\modulis_{\sigma_{t,a}}(\bar{v}_1)$ and $\theta_{\bar{v}_2}(\bar{v}_1)$ is an ample class on $\modulis_{\sigma_{t,a}}(\bar{v}_2)$ using the conventions in \cite{BM14,BM14-MMP}.
\end{remark}

This theorem provides an explicit K3-theoretic modular description for some special double EPW sextics, which may be helpful for our understanding.

In general, we make the following conjecture.

\begin{conjecture}\label{13_conjecture-1}
Let $S$ be a polarized K3 surface in $\MBE(2t)$ with $t\geq 5$ such that there is no K3 surface $S'$ in $\MBE(4)$ satisfying $\cate{D}^b(S)\cong\cate{D}^b(S')$, then the moduli spaces $\modulis_{\sigma_{t,a}}(\bar{v}_1)$ and $\modulis_{\sigma_{t,a}}(\bar{v}_2)$ of semistable objects can be realized as double EPW sextics for some $a$ satisfying $a^2\equiv-1\pmod{t}$.
\end{conjecture}

\begin{remark}
This conjecture has been partially verified in \cite{LZ26} for K3 surfaces in Proposition \ref{12_main-result-2.1}~\eqref{12_main-result-2.1(1)} using \cite{DK24}. Also, the conjecture is partly checked for a general K3 surface in $\MBE(2t)_a$ with $a^2-t=-1$ and $t\geq 10$ by \cite[Corollary 7.6]{DM19}. 
\end{remark}

It is related to the conjecture on the image of the period map for Gushel--Mukai fourfolds, see \cite[Question 9.1]{DIM15}. Moreover, Conjecture \ref{13_conjecture-1} could be seen as evidence for the following stronger conjecture.

\begin{conjecture}
Let $S$ be a K3 surface in $\MBE(2t)$ with $t\geq 10$ such that there is no K3 surface $S'$ in $\MBE(4)\cup\MBE(10)$ satisfying $\cate{D}^b(S)\cong\cate{D}^b(S')$, then there exists a first-type Hodge-special Gushel--Mukai fourfold $X$ such that there exists an $\Stwo$-equivariant equivalence $\cate{D}^b(S)\cong\schc_X$.
\end{conjecture}

This conjecture, together with Conjecture \ref{13_conjecture-1}, would provide better understanding of the period image of double EPW sextics (see \cite[Example 6.3]{DM19}).

{
\appendix
\section{Bridgeland--Enriques realization}\label{000_Section-A}
The author's advisor Paolo Stellari suggests that it would also be interesting to consider the following general notions.

\begin{definition}
Let $S$ be a K3 surface. Then it is said to be \defi{Enriques realizable} if there exists an $\Stwo$-action on $\cate{D}^b(S)$ which is generated by a non-symplectic derived involution $\Pi$; an Enriques realizable polarized K3 surface is said to be \defi{Bridgeland} if the involution $\Pi$ fixes a stability condition in $\Stab^{\dag}(S)$.
\end{definition}

The known Enriques realizable K3 surfaces include Bridgeland--Enriques general K3 surfaces, the K3 double covers over smooth projective surfaces (e.g.~the K3 coverings over Enriques surfaces and the branched double cover over Hirzebruch surfaces \cite{Ha21,Ha24}), and the K3 surfaces in Example \ref{11_t=2-example1} and Example \ref{11_t=2-example2}. 

The second class of known Enriques realizable K3 surfaces is Bridgeland, because the non-symplectic involutions are induced by automorphisms on surfaces. Also, the two extra examples are Bridgeland--Enriques realizable as follows.

\begin{example}
Let $S$ be a very general K3 surface in Example \ref{11_t=2-example1} such that the class $H_1$ is ample. Then one checks directly that
$$\sigma_{c,m}:=\sigma_{cH_1,mH_1-\frac{1}{2}H_2}$$
is a stability condition on $\cate{D}^b(S)$ for any $c>0$ and $m\in \br- \bq$ and the Hodge isometry $\tau_1$ fixes its central charge. By construction, one has
$$\Pi_1(\sho_x)\cong\sho_{\iota_1(x)}$$
for any $x\in S$, where $\iota_1$ is the covering involution of the degree-two morphism $S\rightarrow\bp^2$ defined by $|H_1|$. It implies that $\Pi_1$ fixes the stability condition $\sigma_{c,m}$.
\end{example}

\begin{example}
Let $S$ be a Picard rank two K3 surface in Example \ref{11_t=2-example2}, then 
$$\sigma_q:=\sigma_{\frac{1}{\sqrt{4-2q^2}}(H-qE),-\frac{3}{2}H}$$
is a stability condition on $\cate{D}^b(S)$ for any $q\in(0,1)$ and the Hodge isometry $\tau$ fixes its central charge. To show that the involution $\Pi$ fixes the stability condition $\sigma_q$, it suffices to show that $\Pi\sho_x$ is $\sigma_q$-stable with phase $1$ for any $x\in S$.

Suppose that $x$ is not in the exceptional divisor $E$, one computes that
$$0\rightarrow \Pi\sho_x[-1]\rightarrow\HH^0(\shi_x(H))\otimes\sho_S(-H)\rightarrow\shi_x\rightarrow0$$
using the expression of $\Pi$ in Example \ref{11_t=2-example2}. Then $\shg_x:=\Pi\sho_x[-1]$ is a vector bundle with Mukai vector $(2,-3H,9)$. Since $Z_q(\shg_x)=1$, it remains to check that $\shg_x$ is slope stable with respect to the ample class $H-qE$.

Otherwise, one can choose a line bundle $\shl\hookrightarrow\shg_x$ whose slope is no less than the slope of $\shg_x$. Since $\shg_x\subset\sho_S(-H)^{\oplus 3}$, one has a non-trivial factorization $\shl\rightarrow\sho_S(-H)$. It follows that $\shl\cong\sho_S(-H-C)$ for some effective divisor $C$.

Set $C=aH+bE$, then one has $a\geq0$ and $4a+2b\geq0$ as $H$ is nef and $D=H-E$ is ample. Since $0<q<1$, one can check that the slope of $\shl$ is strictly smaller than that of $\shg_x$ if we assume in addition that $a\geq 1$. Consequently, one must have $a=0$ and then $\shl\cong\sho_S(-H-bE)$ for $b\geq0$. It is also impossible. In fact, one can use the short exact sequence defining $\shg_x$ to show that $\HH^0(\shg_x(H+bE))=0$.

Suppose that $x$ is contained in $E$, one can find two exact sequences
$$0\rightarrow \shh\rightarrow\HH^0(\sho_S(D))\otimes\sho_S\rightarrow\sho_S(D)\rightarrow0$$
$$0\rightarrow \Pi\sho_x[-1]\otimes\sho_S(H)\rightarrow\shh\rightarrow\sho_E(-x)\rightarrow0$$
using the expression of $\Pi$ in Example \ref{11_t=2-example2}. Then $\shg_x:=\Pi\sho_x[-1]$ is a vector bundle with Mukai vector $(2,-3H,9)$. Since $Z_q(\shg_x)=1$, it remains to check that $\shg_x$ is slope stable with respect to the ample class $H-qE$.

As before, any possible destabilizing sheaf of $\shg_x$ has the form $\sho_S(-H-bE)$ for some $b\geq0$. However, one can use the exact sequences to check 
$$\HH^0(\shg_x(H+bE))=\HH^0(\shh(bE))=0$$ 
which leads to a contradiction.	
\end{example}

So it is reasonable to make the following conjecture.

\begin{conjecture}\label{A_conjecture}
An Enriques realizable K3 surface is Bridgeland.
\end{conjecture}

To attack this conjecture, one needs to first exclude the possible $(-2)$-roots in the anti-invariant lattice for some involutive isometries. The classification of the qualified Hodge involutions has already been a difficult question even under the additional assumption that the K3 surface is pure in the following sense.

\begin{definition}
An Enriques realizable K3 surface $(S,\Pi)$ is said to be \defi{pure} if there exists an ample class $H$ on $S$ such that the Hodge isometry induced by $\Pi$ preserves the sublattice $\Omega_{S,H}:=\HH^0(S,\bz)\oplus\bz[H]\oplus\HH^4(S,\bz)$.
\end{definition}

The Bridgeland--Enriques general K3 surfaces and the K3 branched double covers over smooth projective surfaces are pure. To find exotic examples is subtle, because one needs to check against all ample polarizations on a K3 surface.

\begin{example}
Consider a K3 surface $S$ with Picard lattice 
$$\Pic(S)=\begin{pmatrix}
0&6\\
6&8
\end{pmatrix}$$
Then
\begin{enumerate} 
	\item $S$ is not a double cover over a smooth projective surface;
	\item $S$ is not Bridgeland--Enriques general for any ample class on $S$; 
    \item $S$ is Bridgeland--Enriques realizable for some ample class $H$ on $S$.
\end{enumerate}

\begin{proof}
(1) Since the Picard rank of $S$ is two, it is not the K3 covering of an Enriques surface. Suppose that $S$ is a branched double cover over a rational surface, then the covering involution on $S$ would be non-symplectic. Then it acts as $-\id$ on the discriminant group of $\Pic(S)$. However, a direct computation on $\Pic(S)$ shows that such an action cannot preserve the ample cone.

(2) Here we choose divisor classes $D$ and $E$ on $S$ satisfying
$$D^2=8,\quad D\cdot E=6,\quad E^2=0$$
then a primitive ample class has the form $A=xD+yE$. Set $k=2x+3y$, then $x$ and $y$ are coprime and one can assume $x,k>0$ by choosing the sign of $E$.

Suppose otherwise that the polarized K3 surface $(S,A)$ is in $\MBE(2t)_a$, then 
$$a^2+1=mt$$
for $t=2kx$ and some integer $m$. According to Theorem \ref{02_FL-result} \eqref{02_FL-result(3)}, the factors $k$ and $x$ are odd and not divisible by $3$. In particular, $x$ and $k$ are also coprime. It further ensures that $t\equiv1\pmod{3}$ as $k\equiv 2x\pmod{3}$.

Similar to Proposition \ref{11_BE-locus-square-2-vector}, one can choose integers $u,v$ such that
$$u^2+v^2=t,\quad 3|v,\quad v\equiv\varepsilon au\pmod{2x},\quad v\equiv-\varepsilon au\pmod{2k}$$
for a suitable sign $\varepsilon =\pm1$. Then one checks that
$$p:=\frac{-au-\varepsilon v }{2k}\quad\textup{and}\quad q:=\frac{-au+\varepsilon v}{6x}+\frac{au+\varepsilon v }{3k}$$
are integers.

In this case, one checks that the divisor class $\Delta:=pD+qE$ will prevent $(S,A)$ from being Bridgeland--Enriques general using Proposition \ref{11_main-result-1-lemma1}.

(3) One checks directly that $H=3D+2E$ is an ample class and we will show that $H$ makes $S$ a pure Bridgeland--Enriques realizable K3 surface. 

Here we take the reflection $\bar{\phi}$ on $\Num(S)$ with respect to $(9,H,9)$ and extend it to a non-symplectic Hodge involution $\phi$ on $\MukaiZ{S}$ by gluing with $-\id_{T(S)}$ on the transcendental lattice $T(S)$. One checks that $\phi$ fixes the function $Z_{\omega_0,\beta_0}$ for $\omega_0=\sqrt{2}H/36$ and $\beta_0=H/9$. It is not difficult to check that the pair $(\omega_0,\beta_0)$ defines a stability condition $\sigma_0$ using Theorem \ref{02_K3-Bridgeland}.

Consequently, there exists a non-symplectic involution $\Pi$ on $\cate{D}^b(S)$ whose Hodge isometry is $\phi$ and which fixes $\sigma_0$. One notices that $\phi(1,0,-1)=(1,0,-1)$, so $\Pi$ induces an involution on the moduli space $X=\modulis_{\sigma_0}(1,0,-1)$. One checks using the wall-crossing criteria in \cite{BM14,BM14-MMP} that $X$ is a $\KK^{[2]}$-type hyperkähler manifold, so $\Pi$ fixes a stable object in $X$ by \cite{Bea11} (similar to Theorem \ref{11_main-result-1}). Then $\Pi$ induces an $\Stwo$ group action on $\cate{D}^b(S)$ according to Proposition \ref{02_Enriques-group-action-cri} and makes $S$ a pure Bridgeland--Enriques realizable K3 surface as it is non-symplectic.
\end{proof}
\end{example}

This complexity is also one of the main obstructions for Bridgeland's conjecture on stability manifolds of K3 surfaces \cite[Conjecture 1.2]{Bri08}. 

On the other hand, one has the following statement which is useful for constructing examples from involutions on hyperkähler manifolds as in \cite{Liu26b}.

\begin{proposition}\label{A_ppst}
Consider a $\KK^{[n]}$-type hyperkähler manifold $X$ which is isomorphic to a moduli space $\modulis_{\sigma}(v)$ for a vector $v$ and a $v$-generic stability condition $\sigma\in\Stab^{\dag}(S)$ on some K3 surface $S$. Let $f\in\Aut(X)$ be an anti-symplectic involution with fixed points and whose action on the discriminant group $A_{\HH^2(X,\bz)}$ is the identity. Then $S$ is a Bridgeland--Enriques realizable K3 surface.

\begin{proof}
One takes an invariant ample class on $X$, which is the image of some vector $w$ in $v^{\perp}\subset\MukaiZ{S}$ under the Mukai isomorphism in Theorem \ref{11_bridgeland-hyperk-K3[n]}. One can see that the function $Z_0\colon\Num(S)\rightarrow\bc$ corresponding to $v-w\bm{i}\in\Num(S)\otimes\bc$ is contained in the space $\mathfrak{P}_0^+(S)$ of central charges in Remark \ref{02_K3-Bridgeland-remark} as \cite[Theorem 12.1]{BM14-MMP} excludes the spherical object obstructions appearing in Theorem \ref{02_K3-Bridgeland}.

Since the isometry $f^*$ on $\HH^2(X,\bz)$ acts on $A_{\HH^2(X,\bz)}$ as the identity, it extends to an involution $\phi$ in $\OO^+(\MukaiZ{S})$ by the standard facts in lattice theory (see for example \cite[Proposition 14.2.6]{HuK3}). The isometry $\phi$ fixes $Z_0$ by construction. So one can find a non-symplectic involution $\Pi$ on $\cate{D}^b(S)$ and some $\sigma_0\in\Stab^{\dag}(S)$ with central charge $Z_0$ such that $\Pi.\sigma_0=\sigma_0$ using Theorem \ref{02_K3-Hu-conway}.

Then we claim that $\sigma_0$ is $v$-generic. By virtue of the Bayer--Macrì wall-crossing criterion in \cite{BM14,BM14-MMP}, the class $\theta_{\sigma_0}(w)$ on $\modulis_{\sigma_0}(v)$ is not in the boundary of the ample cone decomposition as $\theta_{\sigma}(w)$ is ample on $X=\modulis_{\sigma}(v)$. Since $\theta_{\sigma_0}(w)$ is exactly the divisor class $\ell_{\sigma_0}$ in \cite{BM14} by construction, one concludes that $\theta_{\sigma_0}(w)$ is ample and the stability condition $\sigma_0$ is $v$-generic. It also ensures that $X\cong\modulis_{\sigma_0}(v)$.

In this case, the action of $\Pi$ on $\modulis_{\sigma_0}(v)$ induces the involution $f$ we start with. So it fixes at least one stable object. Consequently, it makes $S$ a Bridgeland--Enriques realizable K3 surface by Proposition \ref{02_Enriques-over-K3-cri} and Proposition \ref{02_Enriques-group-action-cri}.
\end{proof}
\end{proposition}

\begin{remark}
Let $f$ be an involution on a $\KK^{[n]}$-type hyperkähler manifold $X$. Then $f$ always has fixed points when $n$ is even by \cite{Bea11}. Suppose that $n$ is odd and $f$ is fixed-point-free, then the quotient $X/\langle f\rangle$ is called an Enriques manifold \cite{OS11}. 
\end{remark}

}

\section*{Acknowledgments}
This article, together with articles \cite{Liu26a,Liu26b,LZ26}, makes up the author's doctoral dissertation in Milan. The author would like to thank his advisors Laura Pertusi and Paolo Stellari for constant support, beneficial conversations, and suggestions. 

Part of this work was carried out while the author was visiting the University of Edinburgh. The author wants to thank Arend Bayer for the invitation and the weekly discussions, and thank the institutions for their warm hospitality.

The author would also like to thank Yu-Wei Fan and Kuan-Wen Lai for answering his questions concerning their work.

\end{document}